\documentclass[12pt,reqno]{article}
\usepackage{color,amsmath,amssymb,lscape,graphicx,ifthen,pdfpages}
\usepackage{color}
\usepackage{fullpage}
\usepackage{float}
\usepackage{psfig}
\usepackage{amsthm}
\usepackage{amsfonts}
\usepackage{latexsym}
\usepackage{epsf}

\newcommand{\seqnum}[1]{\href{http://oeis.org/#1}{\underline{#1}}}
\usepackage[colorlinks=true,
linkcolor=webgreen,
filecolor=webbrown,
citecolor=webgreen]{hyperref}
\definecolor{webgreen}{rgb}{0,.5,0}\definecolor{webbrown}{rgb}{.6,0,0}

\begin{document}

\theoremstyle{plain}
\newtheorem{theorem}{Theorem}
\newtheorem*{theorem*}{Theorem}
\newtheorem{corollary}[theorem]{Corollary}
\newtheorem{lemma}[theorem]{Lemma}
\newtheorem{proposition}[theorem]{Proposition}
\theoremstyle{definition}
\newtheorem{definition}[theorem]{Definition}
\newtheorem{example}[theorem]{Example}
\newtheorem{conjecture}[theorem]{Conjecture}
\theoremstyle{remark}
\newtheorem{remark}[theorem]{Remark}
\newtheorem{case}{Case}

\def\glqq{,\,\!\!,}
\def\grqq{`\,\!`}
\def\eg{{\it e.g.},\,}
\def\Eg{{\it E.g.},\,}
\def\ie{{\it i.e.},\,}
\def\Ie{{\it I.e.},\,}
\def\viz{{\it viz}\ }
\def\etc{\,{\it etc.}\,}
\def\via{{\it via}\, }
\def\sspp{\,+\,}
\def\sspm{\,-\,}
\def\sspeq{\,=\,}
\def\sspeqmust{\buildrel! \over =}
\def\sspdef{\, :=\,}
\def\sspneq{\,\neq\,}
\def\sspkl{\,<\,}
\def\sspgr{\,>\,}
\def\sspgeq{\,\geq\, }
\def\sspleq{\,\leq\,} 
\def\spgeq{\,\geq\, }
\def\spleq{\,\leq\,} 
\def\pn{\par\noindent}
\def\pb{\par\bigskip}
\def\ps{\par\smallskip}
\def\pbn{\par\bigskip\noindent}
\def\psn{\par\smallskip\noindent}
\def\sspentsp{\, \hat =\, }
\def\sspequiv{\,\equiv\,}
\def\sspnotequiv{\,\not\equiv  \,}
\def\sspmapsto{\,\mapsto\,}
\def\sspFollows{\,\Rightarrow\,}
\def\sspdiv{\,|\,}
\def\Beq{\begin{equation}}
\def\Eeq{\end{equation}}
\def\Beqarray{\begin{eqnarray}}
\def\Eeqarray{\end{eqnarray}}
\def\sspin{\,\in\,}
\def\sspfed{\,=:\,}
\def\sspto{\,\to\,}
\def\rhs{right-hand side\,}
\def\lhs{left-hand side\,}
\def\dstyle#1{$\displaystyle #1 $}
\def\union{\cup}
\def\floor#1{\left\lfloor{#1}\right\rfloor}
\def\ceil#1{\left\lceil{#1}\right\rceil}
\def\abs#1{\vert {\,#1\,} \vert}
\def\Cases2#1#2#3#4{\left\{\begin{array}{ll}#1&\mbox{#2}\\ &\\#3&\mbox{#4}\end{array}\right.}
\def\binomial#1#2{{#1} \choose {#2}}
\def\sspdiv{\,|\,} 
\def\rprod{\prod\llap {\raise 8pt\hbox{$\rightarrow \thinspace$}}} 
\def\range#1#2{#1,\,\count15=#1 \advance\count15 by +1 \number\count15,\,...,\,#2\,}
\def\rangeinf#1{#1,\, \count16=#1 \advance\count16 by +1 \number\count16,\,...\,}
\font\smallcmr=cmr6
\def\cProd{\rlap{$\>\,\!${\raise2pt\hbox{\smallcmr c}}}\Pi}
\def\vjA{\overrightarrow{j_A}}
\def\vjANp{\overrightarrow{j_A}(N,p)}
\def\vjC{\overrightarrow{j_C}}
\def\vjCNp{\overrightarrow{j_C}(N,p)}
\def\TWord{TW\ \llap{$o$}rd}
\def\Tseq{T\,\llap{$s$}eq}
\def\Nseq{N\,\llap{$s$}eq}
\def\fseq{f\,\llap{$s$}eq}
\rightline{revised version, September 21 2020}
\vbox {\vspace{6mm}}
\begin{center}
\vskip 1cm{\LARGE\bf 
The Tribonacci and $\bf ABC$ Representations of Numbers are Equivalent 
}
\vskip 1cm
\large
Wolfdieter L a n g \footnote{ \url{http://www.itp.kit.edu/~wl}} \\
Karlsruhe \\
Germany\\
\href{mailto:wolfdieter.lang@partner.kit.edu}{\tt wolfdieter.lang@partner.kit.edu}
\end{center}

\vskip .2 in
\begin{abstract}
It is shown that the unique representation of positive integers in terms of tribonacci numbers and the unique representation in terms of iterated $A$, $B$ and $C$ sequences defined from the tribonacci word are equivalent. Two auxiliary representations are introduced to prove this bijection. It will be established directly on a node and edge labeled tribonacci tree as well as formally. A systematic study of the $A$, $B$ and $C$ sequences in terms of the tribonacci word is also presented.
\end{abstract}

\bigskip

\section{Introduction}
The quintessence of many applications of the tribonacci sequence $T\sspeq\{T(n)\}_{n=0}^{\infty}$ \cite{OEIS} \seqnum{A000073} \cite{Weisstein1}, \cite{Wiki1}, \cite{Carlitz} \cite{Barcucci} is the ternary substitution sequence $2\sspto 0$,  $1\sspto 02$ and  $0\sspto 01$. Starting with $2$ this generates an infinite (incomplete) binary tree with ternary node labels called $TTree$. See {\sl Fig 1} for the first $6$ levels $l\sspeq \range{0}{5}$ denoted by $TTree_5$. The number of nodes on level $l$ is the tribonacci number $T(l+2)$, for $l\sspgeq 0$. In the limit $n\sspto \infty$ the last level $l \sspeq n$ of $TTree_n$ becomes the infinite self-similar tribonacci word $\TWord$. The nodes on level $l$ are numbered by $N\sspeq \range{0}{T(l+2)\sspm 1}$.
\psn
The left subtree, starting with $0$ at level $l\sspeq 1$ will be denoted by $TTreeL$, and the right subtree, starting with $2$ at level $l\sspeq 0$ is named $TTreeR$.  The node $0$ at level $l\sspeq 1$ belongs to both subtrees. The number of nodes on level $l$ of the left subtree $TTreeL$ is $T(l+1)$, for $l\sspgeq 1$; the number of nodes on level $l$ of $TTreeR$ is $1$ for $l\sspeq 0$ and $l\sspeq 1$ and $T(l+2)\sspm T(l+1)$, for $l\sspgeq 2$.
\psn
$\TWord$ considered as ternary sequence $t$ is given in \cite{OEIS} \seqnum{A080843} (we omit the OEIS reference henceforth if $A$ numbers for sequences are given): $\{0,\, 1,\,0,\,2,\,0,\,1,\,0,\,0,\,1,\,0,\,2,\,0,\,1,\,...\}$, starting with with $t(0)\sspeq 0$.  See also {\it Table 1}. This is the analogue of the binary rabbit sequence \seqnum{A005614} in the {\sl Fibonacci} case. Like in the {\sl Fibonacci} case with the complementary and disjoint {\sl Wythoff} sequences $A\sspeq$ \seqnum{A000201}  and $B\sspeq$ \seqnum{A001950} recording the positions of $1$ and $0$, respectively, in the tribonacci case the sequences $A\sspeq$ \seqnum{A278040}, $B\sspeq$ \seqnum{A278039}, and $C\sspeq$ \seqnum{A278041} record the positions of $1$, $0$, and $2$, respectively. These sequences start with $A \sspeq \{1,\,5,\,8,\,12,\,14,\,18,\,21,\,25,\, 29,\, 32,\, ...\}$, $B \sspeq \{0,\,2,\,4,\,6,\,7,\,9,\,11,\,13,\, 15,\,17,\, ...\}$, and $C \sspeq \{3,\,10,\,16,\,23,\,27,\,34,\,40,\,47,\, 54,\,60,\, ...\}$. The offset of all sequences is $0$. See also {\it Table 1}. 
\psn
The present work is a generalization of the theorem given in the {\sl Fibonacci} case for the equivalence of the {\sl Zeckendorf} and {\sl Wythoff} representations of numbers in \cite{WLang}.
\psn 
Note that there are other complementary and disjoint tribonacci $A$, $B$ and $C$ sequences given in OEIS. They use the same ternary sequence $t \sspeq$ \seqnum{A080843} (which has offset 0), with $0 \sspto a$, $1\sspto b$ and $2\sspto c$, however with offset 1, and record the positions of $a$, $b$ and $c$ by A = A003144, B = A003145 and C = A003146, respectively. In \cite{Carlitz} and \cite{Barcucci} they are called $a$, $b$, and $c$. This tribonacci $ABC$-representation is given in \seqnum{A317206}. The relation between these sequences (we call them now $a$, $b$, $c$) and the present one is: $a(n) \sspeq  B(n-1) \sspp 1$, $b(n)\sspeq A(n-1) \sspp 1$, and $c(n) \sspeq C(n-1) + 1$, for $n \sspgeq 1$. We used $B(0) = 0$ in analogy to the {\sl Wythoff} representation in the {\sl Fibonacci} case. The ternary triboonacci sequence appears also with offset $1
$ and entries $1,\,2$ and $3$ as \seqnum{A092782}.  
\psn
From the uniqueness of the ternary sequence $t$ it is clear that the three sequences $A$, $B$ and $C$ cover the nonnegative integers $\mathbb N_0$  completely, and they are disjoint. In contrast to the {\sl Fibonacci} case where the {\sl Wythoff} sequences are {\sl Beatty} sequences \cite{Wiki2} for the irrational number $\varphi\sspeq$\seqnum{A001622}, the golden section, and are given by $A(n)\sspeq \floor{n\,\varphi}$ and $B(n)\sspeq \floor{n\,\varphi^2}$, for $n\sspin \mathbb N$ (with $A(0) \sspeq 0 \sspeq B(0)$), no such formulae for the complementary sequences $A$, $B$ and $C$ in the tribonacci case are considered. The definition given above in terms of $\TWord$, or as sequence $t$, is not burdened by numerical precision problems.\psn
Note that the irrational tribonacci constant $\tau\sspeq 1.83928675521416...\sspeq$\seqnum{A058265}, the real solution of characteristic cubic equation of the tribonacci recurrence $\lambda^3 \sspm \lambda^2\sspm \lambda \sspm 1\sspeq 0$, defines, together with \dstyle{\sigma\sspeq \frac{\tau}{\tau\sspm 1}\sspeq 2.19148788395311... \sspeq}\seqnum{A316711} the complementary and disjoint {\sl Beatty} sequences $At \sspdef \floor{n\,\tau}$ and $Bt \sspdef \floor{n\,\sigma}$, given in \seqnum{A158919 } and \seqnum{A316712}, respectively.
\psn
\parbox{16cm}{\begin{center}
{\includegraphics[height=7cm,width=1\linewidth]{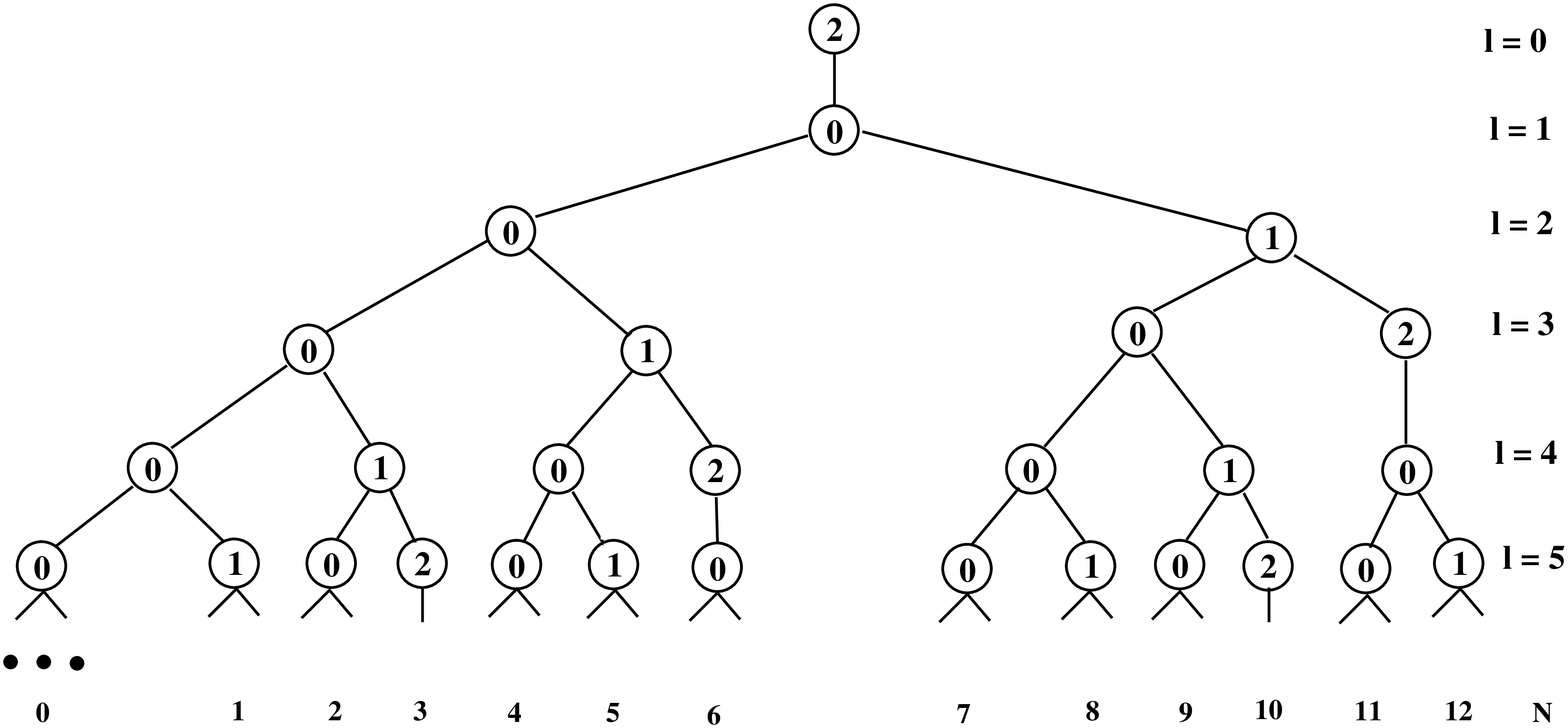}}
\end{center}
}
\psn
\hskip 5cm {\bf  Figure 1: Tribonacci Tree $\bf TTree_5$} 
\pbn
The analogue of the unique {\sl Zeckendorf} representation of positive integers is the unique tribonacci representation of these numbers. 
\Beq \label{NT}
(N)_T \sspeq \sum_{i=0}^{I(N)}\,f_i(N)\,T(i+3) , \ \ f_i(N)\sspin \{0,\,1\},\ \ f_i(N)\,f_{i+1}(N)\, f_{i+2}(N)\sspeq 0,\,\, f(N)_{I(N)}\sspeq 1\,.
\Eeq
The sum should be ordered with falling $T$ indices. This representation will also be denoted by ($Z$ is used as a reminder of {\sl Zeckendorf})
\Beqarray \label{ZTN}
ZT(N)&\sspeq& \cProd_{i=0}^{I(N)}\,f(N)_{I(N)-i} \nonumber \\ 
&\sspeq& f(N)_{I(N)}f(N)_{I(N)-1}...f(N)_{0}\, . 
\Eeqarray
The product with concatenation of symbols is here denoted by $\cProd$, and the concatenation symbol $\circ$ is not written. This product has to be read from the right to the left with increasing index $i$. This representation is given in \seqnum{A278038}$(N)$, for $N\sspgeq 1$. See also {\sl Table 2} for $ZT(N)$  for $N\sspeq \range{1}{100}$.\psn
\Eg $(1)_T\sspeq T(3),\, ZT(1)\sspeq 1$; $(8)_T\sspeq T(6)\sspp T(3),\, ZT(8)\sspeq 1001$. The length of $ZT(N)$ is $\#ZT(N)\sspeq I(N)\sspp 1\sspeq \{1,\,2,\,2,\,3,\, 3,\,3,\,4,\,4,\,4,\,4,\,4,\,4,\,...\}\sspeq \{$\seqnum{A278044}$(N)\}_{N\sspgeq 1}$. The number of occurrences of each $n\sspin \mathbb N$ in this sequence is given by $\{1,\,2,\,3,\,6,\, 11,\, 20,\, 37\,...\}\sspeq \{$\seqnum{A001590}$(n+2)\}_{n\sspgeq 1}$. These are the companion tribonacci numbers of $T\sspeq$\seqnum{A000073} with inputs $0,\,1,\,0$ for $n\sspeq 0,\,1,\,2$, respectively.
\psn
$ZT(N)$ can be read off any finite $TTree_n$ with $T(n+2)\sspgeq N$ after all node labels $2$ have been replaced by $1$. See {\sl Figure 1} for $n=5$ (with $2\sspto 1$) and numbers $N\sspeq \range{1}{12}$. The branch for $N$ is read from bottom to top, recording the labels of the nodes, ending with the last $1$ label. Then the obtained binary string is reversed in order to obtain the one for $ZT(N)$. \Eg $N\sspeq 9$ leads to the string $0101$ which after reversion becomes $ZT(9)\sspeq 1010$.\psn
The analogue of the {\sl Wythoff} $AB$ representation of nonnegative integers in the {\sl Fibonacci} case is the tribonacci $ABC$ representation using iterations of the sequences $A$, $B$ and $C$. 
\Beq \label{NABC}
(N)_{ABC}\sspeq \left(\cProd_{j=1}^{J(N)}\, X(N)_j^{k(N)_j}\right)B(0),\ \text{with} \ \ N\sspin \mathbb N_0, \ \ k(N)_j\sspin \mathbb N_0\,, 
\Eeq
again with an ordered concatenation product. Here $X(N)_j\sspin \{A,\,B,\,C\}$, for \pn
$j\sspeq \range{1}{J(N)-1}$, with  $X(N)_j \sspneq X(N)_{j+1}$, and $X(N)_{J(N)} \sspin \{A,\,C\}$. Powers of $X(N)_j$ are also to be read as concatenations. Concatenation means here iteration of the sequences. The exponents can be collected in $\vec k(N)\sspdef (k(N)_1,\, ...,\,k(N)_{J(N)})$.
For the equivalence proof only positive integers $N$ are considered. If exponents vanish the corresponding $A,\,B,\,C$ symbols are not present ($X(N)_j^0$ is of course not $1$).  If all exponents vanish, the product $\cProd$ is empty, and $N\sspeq 0$ could be represented by $(0)_{ABC} \sspeq B(0) = 0$ (but this will not be used for the equivalence proof). Each $ABC$ representation ends in a single $B$ acting on $0$ because the argument of $X\sspin \{A,B,C\}$ for $N\sspeq 1, 2$ and $3$ is $0$ written as $B(0)$, and for other $N$ the iterative tracking of the arguments always leads back to one of these three numbers. \Eg for $N\sspeq 15$ these arguments are $8$ (for $X \sspeq B$), $2$ (for $A$) and $1$ (for $A$) and $1\sspeq B(0)$, hence $15 \sspeq B(A(B(A(B(0)))))$, later abbreviated as $01010$ (see {it Table 3}).  \psn
\Eg $(30)_{ABC} \sspeq (BCBA)B(0)\sspeq B(C(B(A(B(0)))))$, $J(30) \sspeq 4$,\ $\vec k(30)\sspdef (k_1,\,k_2,\,k_3,\,k_4)$ $\sspeq (1,\,1,\,1,\,1)$,  $X(30)_1^{k_1} \sspeq B^1,\ X(30)_2^{k_2} = C^1, X(30)_3^{k_3} \sspeq B^1,\ X(30)_4^{k_4} = A^1$ (sometimes the arguments $(N)$ are skipped).\psn
The number of $A,\, B$ and $C$ sequences present in this representation of $N$ is $\sum_{j=1}^{J(N)}k(N)_j \sspp 1 \sspeq$\seqnum{A316714}$(N)$, This representation is also written as
\Beq \label{NABC012}
ABC(N)\sspeq \left(\cProd_{j=1}^{J(N)}\, x(N)_j^{k(N)_j}\right)0, \text{with} \ \ k(N)_j\sspin \mathbb N_0\,, 
\Eeq
and $x(N)_j\sspin \{0,\,1,\,2\}$, for $j\sspeq \range{1}{J(N)-1}$, with $x(N)_j\sspneq x(N)_{j+1}$, and $x(N)_{J(N)}\sspin \{1,\,2\}$. Here $x \sspeq 0,\,1,\,2$ replaces $X\sspeq B,\,A,\,C$, respectively. 
\psn
\Eg $ ABC(0) \sspeq 0,\ J(0)\sspeq 0$ (empty product); $ABC(30)\sspeq 02010$. \psn  
For this $ABC$ representation see \seqnum{A319195}. Another version is \seqnum{A316713} (where for a technical reason $B,\, A$, and  $C$ are represented by $1,\,2$ and $3$ (not $0,\,1$ and $2$), respectively). See also {\sl Table 3} for $ABC(N)$, for $N\sspeq \range{1}{100}$.\psn
The number of $B$s, $A$s and $C$s in the $ABC$ representation of $N$ is given in sequences \seqnum{A316715}, \seqnum{A316716} and \seqnum{A316717}, respectively. As already mentioned the length of this representation is given in \seqnum{A3167174}.
\pbn
\section{Equivalence of representations}
The unique representation $(N)_T$ of $N\sspin \mathbb N$ in terms of tribonacci numbers given in eq.~\ref{NT}, or equivalently as a representation as binary word $ZT(N)$ in eq.~\ref{ZTN}, and the unique representation $N_{ABC}$ as composition of $A,\,B,\,C$ sequences given in eq.~\ref{NABC} will be shown to be equivalent, \ie they can be transformed into each other without computing the actual value of $N$. This will be done in two directions with the help of two auxiliary representations denoted by $\widehat{ZT}(N)$ and $(N)_{AB\bullet\boldsymbol\times}$. See the {\it Figure 2} for the pictogram of the transformations (mappings) between these representations in the anti-clockwise and clockwise directions. The uniqueness of the two representations $ZT(N)$ and $(N)_{ABC}$ will be proved in {\it section 3}. \psn
\psn \parbox{16cm}{\begin{center}
{\includegraphics[height=7cm,width=.6\linewidth]{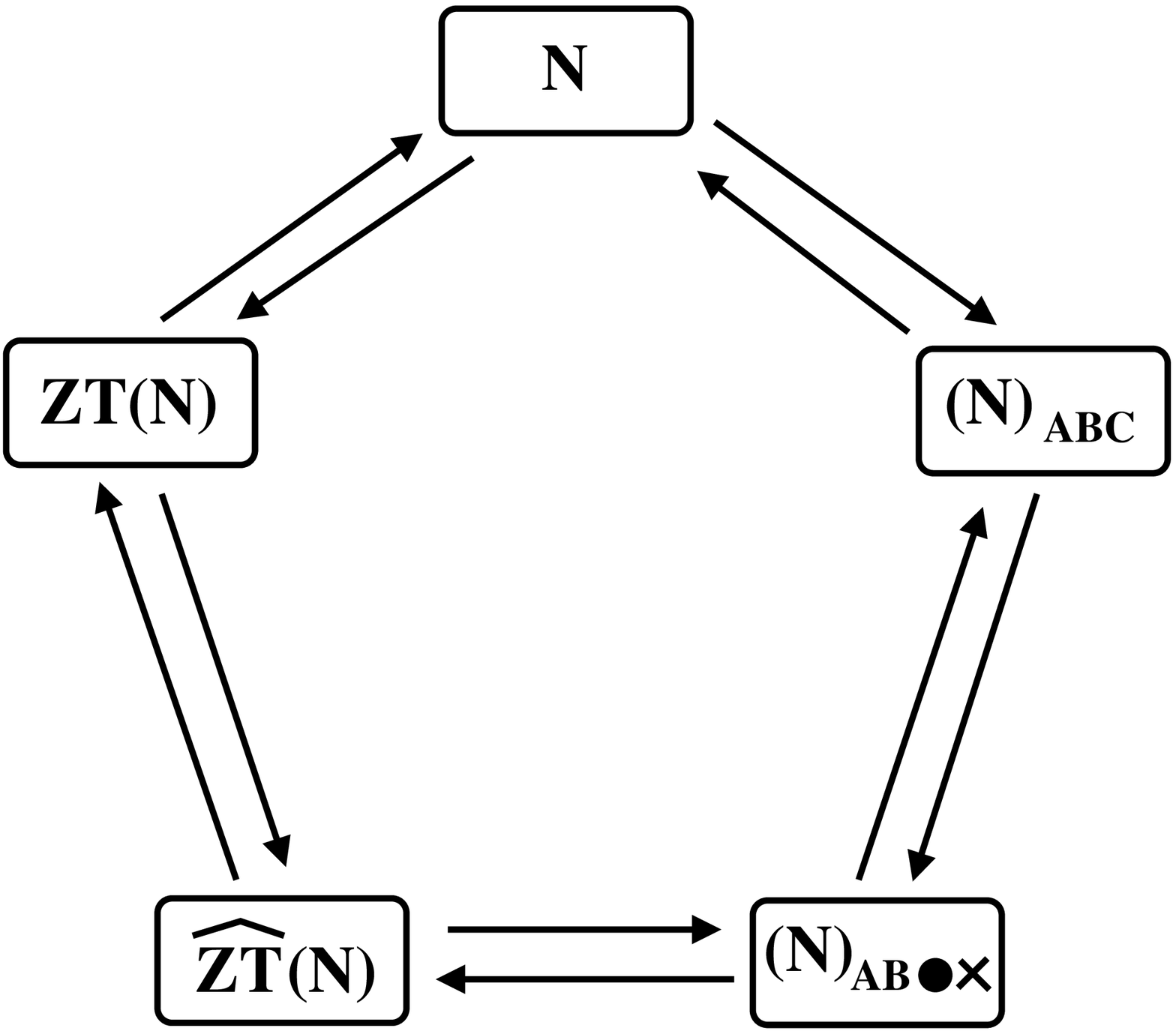}}
\end{center}
}
\psn
\hskip 3cm {\bf  Figure 2: Equivalence of representations} 
\pbn
{\bf I) From $\bf ZT(N)$ to $\bf (N)_{ABC}$}
\pbn
{\bf A) From $\bf ZT(N)$ to $\bf \widehat{ZT}(N)$}
\begin{definition} \label{mapA}
The binary word $ZT(N)$ of length \seqnum{A27804}$(N)$, for $N\sspin \mathbb N$, is mapped to  
a binary word $\widehat{ZT}(N)\sspdef 0\overline{ZT(N)}0$ of length \seqnum{A27804}$(N) \sspp 2$, where $\overline{ZT(N)}$ is the reversed word $ZT(N)$.
\end{definition}
\pn
\Eg $\widehat{ZT}(1)\sspeq 010$ from $ZT(1) = 1$; $\widehat{ZT}(30)\sspeq 00110010$ from $ZT(30) \sspeq 100110$.
\psn
{\bf B) From $\bf\widehat{ZT}(N)$ to $\bf (N)_{AB\bullet\boldsymbol\times}$}
The following six substitution rules for all entries of $\widehat{ZT}(N)$ except the last one are used in the following.
\begin{definition} \label{Subsrules} {\bf The six substitution rules}
\Beqarray
(S{\bf\underline{0}}1)\ \ \ \ \ &\phantom{x}{\bf \underline{0}}0\phantom{x\ }:& \ {\bf \underline{0}}\,\longrightarrow\, B \,, \nonumber \\
(S{\bf\underline{0}}2)\ \ \ \ \ &\phantom{xx}{\bf \underline{0}}11\phantom{x\ }:& \ {\bf \underline{0}}\,\longrightarrow\, \boldsymbol{\times}\,, \nonumber\\
(S{\bf\underline{0}}3)\ \ \ \ \ &\phantom{xx}{\bf \underline{0}}10\phantom{x\ }:& \ {\bf \underline{0}}\,\longrightarrow\, A\,, \nonumber\\
(S{\bf\underline{1}}1)\ \ \ \ \ &\phantom{xx}{\bf \underline{1}}1\phantom{x\ }\ \ :& \ {\bf \underline{1}}\,\longrightarrow\, \boldsymbol{\times}\,, \nonumber\\
(S{\bf\underline{1}}2)\ \ \ \ \ &\phantom{xx}{\bf \underline{1}}0x\phantom{x\ }:& \ {\bf \underline{1}}\,\longrightarrow\, \bullet\,, \ \text{for}\, x\sspin \{0,1\}, \nonumber\\
(S{\bf\underline{1}}3)\ \ \ \ \ &\phantom{xx}{\bf \underline{1}}0\emptyset\phantom{x\ }:& \ {\bf \underline{1}}\,\longrightarrow\, B\,. \nonumber
\Eeqarray
\end{definition}
\pn
The boldface and underlined $0$ and $1$ entries of $\widehat{ZT}(N)$ are substituted depending on which entries follow. The empty word symbol $\emptyset$ marks the end of this binary word. The last entry $0$ of $\widehat{ZT}(N)$ is not substituted (it disappears), therefore the resulting  $(N)_{AB\bullet\boldsymbol\times}$, a word over the alphabet of the four letters $A,\, B,\, \bullet,\,\boldsymbol{\times}$, ends always in a single $B$.  
\psn
\Eg $\widehat{ZT}(N)\sspeq 0101010011010$ with $\#\widehat{ZT}(N)\sspeq 13$ translates to $(N)_{AB\bullet\boldsymbol\times}\sspeq A{\bullet}A{\bullet}A{\bullet}B{\boldsymbol \times}{\boldsymbol \times}{\bullet}AB$ with length $\#(N)_{AB\bullet\boldsymbol\times}\sspeq 12$. 
The value of the number $N$ is here not relevant.
\psn
These substitution rules suffice and are not in conflict which each other. 
In $\widehat{ZT}(N)$ no three consecutive $1$s appear and the end is always $10$.
All possible entries of $\widehat{ZT}(N)$, except the last $0$, can be substituted.\pn
\Eg $1\underline{\bold 1}0$ is not needed, because if this is not the end of $\widehat{ZT}(N)$ the substitution applies to $1\underline{\bold 1}00$ or $1\underline{\bold 1}01$, \ie $(S\underline{\bf 1}2)$ in both cases. If it appears at the end substitution $(S\underline{\bf 1}3)$ applies which produces $\boldsymbol\times B$, because the leading $1$ had to be substituted by $(S\underline{\bf 1}1)$. 
\psn
A list of the $20$ triplets, called tribons, which can appear in $(N)_{AB\bullet\boldsymbol\times}$ is given in the next lemma. The multiplicity of each tribon is shown by giving the corresponding underlined triplets of $\widehat{ZT}(N)$ and their following two entries (if also $\emptyset$ is used to signal the end). To distinguish these quintets the $1$ which in the tree $TTree$ (see {\it Figure 1}) is actually a $2$ is given in boldface and red. A vertical bar signals the end of a  $(N)_{AB\bullet\boldsymbol\times}$ representation. \psn
\begin{lemma} \label{20tribons} {\bf The $\bf 20$ tribons}
\pn
First member $A$ (three tribons, multiplicities $1,\,3,\,1$):
\Beq
A\bullet A\,:\ \ \underline{010}10\,; \ \ A\bullet B\,:\ \ \underline{010}00,\ \ \underline{010}01,\ \ \underline{010}0{\color{red}{\bold 1}}\,; \ \ A\bullet \boldsymbol{\times}\,: \underline{010}{\color{red}{\bold 1}} \nonumber\\
\Eeq
First member $B$ (six tribons, multiplicities $3,\,1,\,1,\,3,\,1,\,1$):\pn
\Beqarray
BBB\,:& \underline{000}00, \ \  \underline{000}01,\ \  \underline{000}0{\color{red}{\bold 1}}\,; &
BBA\,: \ \ \underline{000}10\,;\ \  BAB|\,: \ \ \underline{001}0\emptyset\,,\nonumber \\
BA\bullet\,:& \underline{001}00, \ \  \underline{001}01,\ \  \underline{001}0{\color{red}{\bold 1}}\,; & BB\boldsymbol{\times}: \ \ \underline{000}{\color{red}{\bold 1}}1\,; \ \ B\boldsymbol{\times\times}\,:\ \ \underline{00{\color{red}{\bold 1}}}10\,.\nonumber
\Eeqarray
First member $\bullet$ (six tribons, multiplicities $3,\,1,\,1\,3,\,1,\,1$):\pn
\Beqarray
\bullet A\bullet\,:& \underline{101}00, \ \  \underline{101}01,\ \  \underline{101}0{\color{red}{\bold 1}}\,; &
\bullet AB|\,: \ \ \underline{101}0\emptyset\,;\ \  \bullet BA\,: \ \ \underline{100}10\,,\nonumber \\
\bullet BB\,:& \underline{100}00, \ \  \underline{100}01,\ \  \underline{100}0{\color{red}{\bold 1}}\,; & \bullet B\boldsymbol{\times}: \ \ \underline{100}{\color{red}{\bold 1}}1\,; \ \ \bullet\boldsymbol{\times\times}\,:\ \ \underline{10{\color{red}{\bold 1}}}10\,.\nonumber
\Eeqarray
First member $\boldsymbol{\times}$ (five tribons, multiplicities $1,\,3,\,1,\,3,\,1$):\pn
\Beqarray
\boldsymbol{\times} \bullet A\,: && \ \underline{{\color{red}{\bold 1}}10}10\,; \ \ \boldsymbol{\times}\bullet B\,: \ \ \underline{{\color{red}{\bold 1}}10}00,\ \  \underline{{\color{red}{\bold 1}}10}01,\ \  \underline{{\color{red}{\bold 1}}10}0{\color{red}{\bold 1}}\,; \nonumber \\
\boldsymbol{\times}\bullet\boldsymbol{\times}\,: &&\ \underline{{\color{red}{\bold 1}}10} {\color{red}{\bold 1}}1\,;\ \ \boldsymbol{\times\times}\bullet\,\ : \ \ \underline{0{\color{red}{\bold 1}}1}00\,,
\ \ \underline{0{\color{red}{\bold 1}}1}01,\, \  \underline{0{\color{red}{\bold 1}}1}0{\color{red}{\bold 1}}\,; \nonumber \\
\boldsymbol{\times\times}B|\,: &&\ \, \underline{0{\color{red}{\bold 1}}1}0\emptyset\,. \nonumber
\Eeqarray
\end{lemma}
\begin{corollary} \label{Coro4}
The possible $11$ doublets in  $(N)_{AB\bullet\boldsymbol\times}$ are:\pn
\Beq
A\bullet,\, AB|,\; BB,\,BA,\,B\boldsymbol{\times};\ \bullet A,\, \bullet B,\,\bullet\boldsymbol{\times};\ \boldsymbol{\times}\bullet,\,\boldsymbol{\times\times},\, \boldsymbol{\times} B|.
\Eeq
\end{corollary}
\psn
The following lemma collects information which is also later used for proving that the map from $\bf \widehat{ZT}(N)$ to $(N)_{AB\bullet\boldsymbol\times}$ is invertible.
\psn
\begin{lemma} \label{Rules} {\bf Rules for $\bf (N)_{AB\bullet\boldsymbol\times}$}
\pn
1) There is no doublet $\bullet\bullet$.\psn
2) $\boldsymbol{\times}$ always appears either as the triplet $\boldsymbol{\times\times}\bullet$ or at the end as the triplet $\boldsymbol{\times\times}B|$. There is no tribon $\boldsymbol{\times\times\times}$ (see above).\psn
3) $\bullet$ always appears either in the doublet $A\bullet$ or in the triplet $\boldsymbol{\times\times}\bullet$.\psn 
4) $A$ always appears either as doublet $A\bullet$ or as doublet $AB|$ at the end.
\end{lemma}
\begin{proof}
1) From the {\it Corollary} ~\ref{Coro4}.\psn
2) $\boldsymbol{\times}$ originates either from the substitution i) $\underline{0}11$ or ii) $\underline{1}1$. Case i) continues as $\underline{0}110x$ with $x\sspin \{0,\,1\}$ or as $\underline{0}110\emptyset$ because no three consecutive $1$s appear in $\widehat{ZT}(N)$. Hence after substitution as $xx\bullet$ or as $xxB|$.
Case ii) continues to the left as $0\underline{1}1$ (no three consecutive $1$s), hence is substituted by $xx\bullet$ or $xxB|$ because to the right follows $0x$ or $0\emptyset$. 
\psn
3) A $\bullet$ originates only from $\underline{\bold 1}0x$ with $x\sspin \{0,\,1\}$. It can be continued to the left as i) $00\underline{\bold 1}0x$, ii) $10\underline{\bold 1}0x$ or iii)  $01\underline{\bold 1}0x$ (no three consecutive $1$s). After substitution i) leads to $BA\bullet$ and iii) to $\boldsymbol{\times\times}\bullet$. Case ii) leads to the tribon $\bullet A\bullet$ and a continuation to the left will find either an $A$ or a $\boldsymbol{\times\times}$, or again a new $\bullet$ on the left which has to be continued, etc.\psn
4) $A$ originates only from $\underline{010}$ and it continues with $x\sspin \{0,\,1\}$ if this is not the end of the representation. In the first case it becomes $A\bullet$, and in the second $AB|$ without a $\bullet$ following $A$.  
\end{proof}
\psn
Note that whenever $\boldsymbol{\times}$ appears in a tribon other than one of the two given in {\it 2}, \ie as $\bullet\boldsymbol{\times\times}$ or as  $\boldsymbol{\times}\bullet\boldsymbol{\times}$, this means that these tribons are continued in both directions to give one of the claimed two combination. Similarly if $A$ is not followed by $\bullet$ or the final $B$, or if $A$ is preceded by a $\bullet$ the continuation to the right or left will show the claimed two proved  combinations in any representation of $N$.
\psn
This $(N)_{AB\bullet\boldsymbol\times}$ representation can be visualized on an edge and node labeled tribonacci tree, called $ABCTree_n$ if the bottom level is $n$. Such a tree can be used for the representation of $N\sspin\{\range{1}{T(n+2)\sspm 1}\}$.  $N\sspeq 0$ could also be represented but in the equivalence proof only positive $N$ is considered. The nodes are labeled like in the tribonacci tree $TTree_n$. The tree is read from the bottom level $l\sspeq n$ upwards. The edges are labeled on both sides. If the labels on both sides coincide only one label for the edge is depicted. The edges $00$, $01$ and $02$ (always upwards directed) are labeled on the \lhs by $B$ and on the \rhs by $\bullet$. An exception is the top edge $02$ between levels $l\sspeq1$ and $l\sspeq0$ which has only a $B$ label for both sides. The edge $10$ has the labels $A$ and $\boldsymbol{\times}$ on the left- and right-hand side, respectively. The outermost left branch has also only edge label $B$ for both sides.\psn
The labeling is done in accordance with the rules of the $(N)_{AB\bullet\boldsymbol\times}$ word given above.
One has to start for the $N$ representation at the bottom level $l\sspeq n$ always with the left-hand label of the first edge. For the next edges going out from from a node, the choice is fixed from the direction from which the previous edge reached the node. If it reached the node from the right-hand side, the label on the right-hand side of the outgoing edge has to be chosen, and similarly for the left-hand side. Because the representation ends always in a single $B$ if an $N$ belongs to the left sub-tree $TTreeL_n$ the path stops after the first edge labeled $B$ on the leftmost branch has been reached. This explains why also the edge between levels $l\sspeq 2$ and $l\sspeq 1$ has to belong to $TTreeL_n$, hence the node $0$ at level $l\sspeq 1$ belongs to both sub-trees. For $N$ belonging to the right sub-tree $TTreeR_n$ the path goes all the way up to level $l\sspeq 0$ with last edge $B$. See {\it Figure 3} for $ABCTree_5$ for numbers $N\sspeq (0),\range{1}{12}$. 
\psn
\Eg $N\sspeq 6$:\ $B\boldsymbol{\times\times}B$, starting with the left-hand side label $B$, then the second 
$\boldsymbol\times$ has to be chosen because the previous edge came from the right-hand side. $N\sspeq 9$:\ $BA\bullet AB$.  
\psn
\psn \parbox{16cm}{\begin{center}
{\includegraphics[height=7cm,width=1\linewidth]{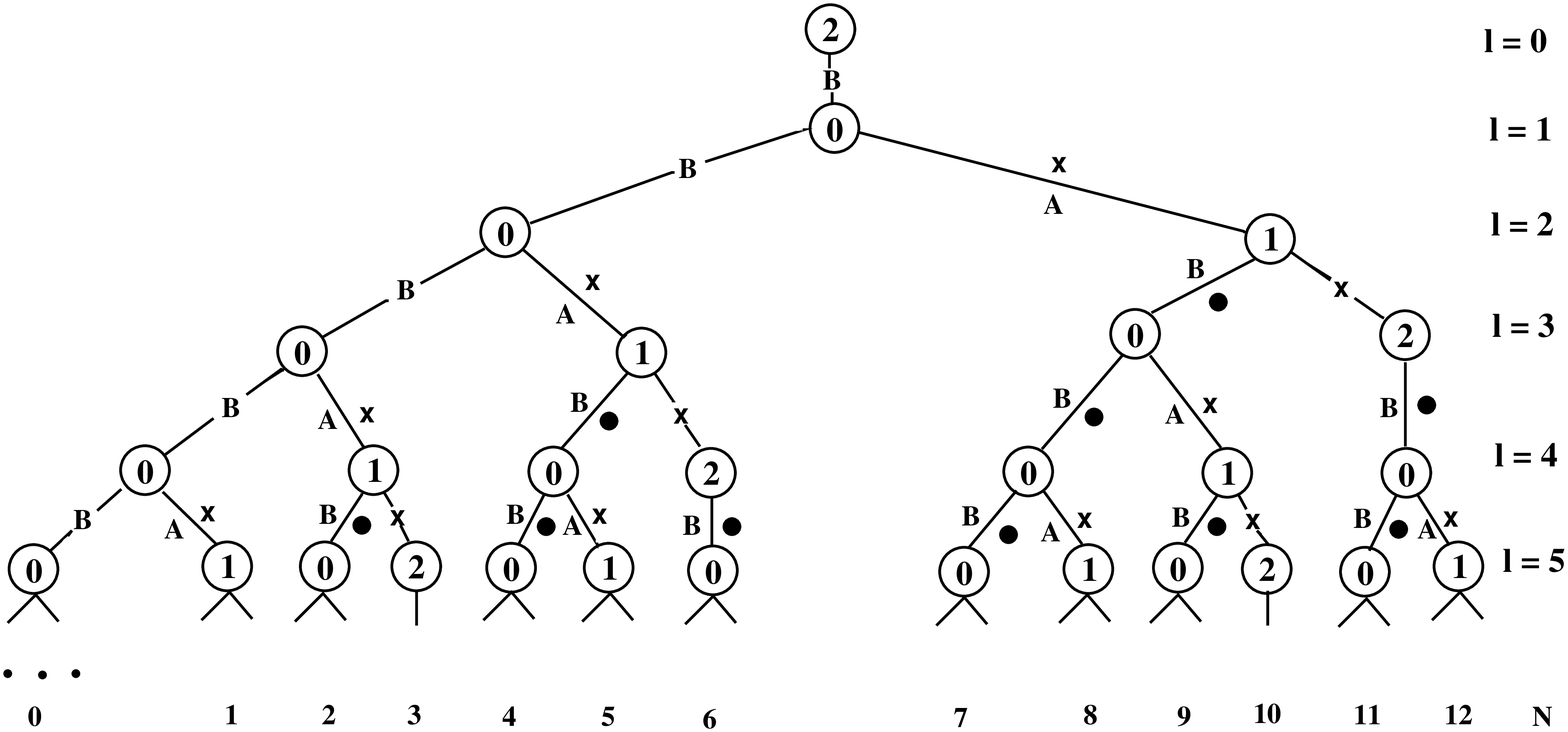}}
\end{center}
}
\psn
\hskip 3cm {\bf  Figure 3: ABC-representation with the $\bf ABCTree_5$} 
\pbn
{\bf C) From $\bf (N)_{AB\bullet\boldsymbol\times}$ to $\bf (N)_{ABC}$} 
\psn
This step is simply performed by the following replacements. It is understood that compositions are meant.\psn
\Beqarray \label{CRep}
&&A\bullet\sspto A,\  \text{and at the end}\ AB \sspto AB(0),\nonumber\\ 
&& \boldsymbol{\times\times}\bullet \sspto C, \text{and at the end}\ \boldsymbol{\times\times}B\sspto CB(0), \nonumber \\
&& B\sspto B. \nonumber
\Eeqarray
\Eg $N\sspeq 6$: $B\boldsymbol{\times\times}B\sspto BCB(0)$, actually $(6)_{ABC}\sspeq B(C(B(0))$. In {\it Table 3} this is denoted as ${ABC}(6)\sspeq 020$, denoting $A,\, B,\,C$ as $1,\,0, 2$. \pn
$N\sspeq 9$:\ $BA\bullet AB\sspto BAAB(0)$, actually $(9)_{ABC}\sspeq B(A(A(B(0))))$. Or ${ABC}(9)\sspeq 0110$.
\pbn
These transformations (mappings) between the representations are now proved to be invertible by reversing the steps.
\psn
{\bf II) From $\bf (N)_{ABC}$ to $\bf ZT(N)$}
\pbn
{\bf $\bf \overline{C})$\ From $\bf (N)_{ABC}$ to 
$\bf (N)_{AB\bullet\boldsymbol\times}$}
\psn
\Beqarray \label{CbarRep}
&&\text{a final}\ AB(0) \sspto AB,\ \text{and}\ A\sspto A\bullet\,,\nonumber \\
&&\text{a final}\ CB(0) \sspto CB,\ \text{and}\ C\sspto \boldsymbol{\times\times}\bullet\,,\nonumber \\
&& B\sspto B\,. \nonumber
\Eeqarray
\Eg $ (N)_{ABC}\sspeq A^3BCAB(0)\sspto  (N)_{AB\bullet\boldsymbol\times}\sspeq A{\bullet}A{\bullet}A{\bullet}B{\boldsymbol \times\times}{\bullet}AB$.
\pbn
{\bf $\bf \overline B)$\ From $\bf (N)_{AB\bullet\boldsymbol\times}$ to $\bf \widehat{ZT}(N)$}
\psn
Two equivalent replacements can be given based on the inversion of the substitution rules of {\it Definition} ~\ref{Subsrules} taking into account the structure of $(N)_{AB\bullet\boldsymbol\times}$.\psn 
{\bf Version 1)} 
\Beqarray  \label{Rep1}
&& \text{at the end}\ \hskip .7cm  AB|\sspto 010,\ \ \hskip 1.3cm A\bullet\sspto 01\,,\nonumber \\
&& \text{at the end}\  \boldsymbol{\times\times}B|\sspto 0110\,,\ \hskip .8cm   \boldsymbol{\times\times}\bullet\sspto 011,\nonumber \\
&& B\sspto 0\,.\nonumber
\Eeqarray
{\bf Version 2)} \psn
Begin with a $0$ and replace
\Beqarray  \label{Rep2}
&& \text{at the end}\hskip .8cm  AB|\sspto 10,\ \ \hskip 1.3cm A\bullet\sspto 10\,,\nonumber \\
&& \text{at the end}\  \boldsymbol{\times\times}B|\sspto 110\,,\hskip .9cm   \boldsymbol{\times\times}\bullet\sspto 110,\nonumber \\
&& B\sspto 0\,.\nonumber
\Eeqarray
This version can be stated as: Begin with a $0$ and replace $A$ and $\boldsymbol{\times}$ by $1$, and $B$ and $\bullet$ by  $0$.
\pbn
\Eg Version 1)
The substitution rules are 
\Beqarray
(N)_{AB\bullet\boldsymbol\times}&\sspeq& A{\bullet}A{\bullet}A{\bullet}B\boldsymbol{\times\times}{\bullet}AB\,,\nonumber\\
\widehat{ZT}(N)&\sspeq& \,01\ 01\ 01\, 0\ \ 0\ 1\ \,1\,0\,1\,0 \,.\nonumber
\Eeqarray
This coincides with the stated replacements.\psn
In version 2) one shifts $(N)_{AB\bullet\boldsymbol\times}$ by one position to the right, producing the leading $0$ in $\widehat{ZT}(N)$.
\Beqarray
(N)_{AB\bullet\boldsymbol\times}&\sspeq& \ \,A{\bullet}A{\bullet}A{\bullet}B\boldsymbol{\times\times}{\bullet}AB\,,\nonumber\\
\widehat{ZT}(N)&\sspeq& \,010\ 10\, 10\ 0\ 1\ \,1\ \ 0\,1\,0 \,.\nonumber
\Eeqarray
\pbn
{\bf $\bf \overline A)$\ From $\bf \widehat{ZT}(N)$ to  $\bf ZT(N)$}
\psn
Omit the two zeros at the beginning and end, and reverse the binary string to obtain $ZT(N)$.
\pbn
\section{Equivalence of representations $\bf ZT(N)$ and $\bf ABC(N)$}
First the uniqueness of the tribonacci-representation $ZT(N)$ of eq.~\ref{ZTN} is considered.\pn
It is clear that every binary sequence starting with $1$, without three consecutive $1$s, represents some 
$N\sspin \mathbb N$. An algorithm for finding such a representation for every $N\sspin \mathbb N$ is given to prove the following lemma.  
\begin{lemma}
The tribonacci-representation $ZT(N)$ of eq.\ref{ZTN} is unique.
\end{lemma}
\noindent
{\bf Proof:}\psn
The recurrence of the tribonacci sequence $T\sspdef \{T(l)\}_{l=3}^\infty$, with inputs $T(3)\sspeq 1,\, T(4)\sspeq 2$ and $T(5)\sspeq 4$, shows that this sequence is strictly increasing. Define the floor function $floor(T;\,n)$, for $n\sspin \mathbb N$, giving the largest member of $T$ smaller or equal to $n$. The corresponding index of $T$ will then be called $Ind(floor(T;\, n))$. Define the finite sequence $\Nseq\sspdef \{N_j\}_{j=1}^{j_{\max}}$ recursively by 
\Beq
N_j\sspeq N_{j-1}\sspm floor(T;\,N_{j-1})\,,\ \ \ \text{for}\ \ j\sspeq \range{1}{j_{\max}},\, 
\Eeq
with $N_0\sspeq N$ and $N_{j_{\max}}\sspeq 0$, but$N_{j_{\max-1}}\sspneq 0$. \psn
It is clear that this recurrence reaches always $0$. Define the finite sequences $fTN\sspdef \{ floor(T;\,N_j)\}_{j=0}^{j_{\max}-1}$ and $IfTN\sspdef\{Ind(fTN_j)\}_{j=0}^{j_{\max}-1}$. Then $I(N)$ in eq.~\ref{ZTN} is given by $I(N)\sspeq IFTN_0$ and the finite sequence $\fseq\sspeq \{f_{I(N)\sspm k}\}_{k=0}^{I(N)}$ is given by
\Beq
f_{I(N)\sspm k} \sspeq {\Cases2{$1$}{$\text{if}\ \ I(N)\sspm k\sspin IfTN\, ,$}{$0$}{\text{otherwise}\ .}}
\Eeq
\hskip 16.1cm $\square$
\pn
{\bf Example 3} $N\sspeq 263$. $\Nseq\sspeq\{263,\,144,\,33,\,9\,2,\,0\}$, $fTN\sspeq \{149,\,81,\,24,\,7,\,2\}$, 
$IfTN\sspeq \{8,\,7,\,5,\,3,\,1\}$, $I(N)\sspeq 8$, $\fseq \sspeq \{1,\,1,\,0,\,1,\,0,\,1,\,0,\,1,\,0\}$.
\pbn
Next follows the lemma on the uniqueness of the $ABC$ representation given in eq.~\ref{NABC} or eq. ~\ref{NABC012}.\psn
\begin{lemma}
The tribonacci $ABC$ representation $(N)_{ABC}$ of eq.\ref{NABC}, for $N\sspin \mathbb N_0$, is unique.
\end{lemma}
\noindent
{\bf Proof:}\psn
From the definition of the $A,\,B$ and $C$sequences (each with offset $0$) based on the value $1,\, 0$ and $2$, respectively, of $t(n)$, for $n \sspin \mathbb N_0$, it is clear that these sequences are disjoint and $\mathbb N_0$-complementary. $0$ is represented by $B(0)$. Therefore the $n$-fold iteration $B^{[n]}(0)$  (written as $B^n(0)$) is allowed only for $n\sspeq 1$, and any representation ends in $B(0)$. Iterations acting on $0$ are encoded by words over the alphabet $\{A,\,B,\,C\}$, and $n$-fold repetition of a letter $X$ is written as $X^n$, named $X-$block, where $n\sspeq 0$ means that no such $X-$block is present. Then any word consisting of consecutive different non-vanishing $X-$blocks ending in the $B-$block $B^1$ represents a number $N\sspin \mathbb N_0$.\pn
In order to prove that with such representations every $N\sspin \mathbb N_0$ is reached the following algorithm is used.
Replace any number $n\sspin \mathbb N_0$, which is $n\sspeq X_n(k(n))$ with $X_n\sspin \{A,\,B,\,C\}$ and $k(n)\sspin \mathbb N_0$, by the $2$-list $L(n)\sspeq [L(n)_1,\,L(n)_2]\sspdef [X_n,\,k(n)]$ . Define the recurrence 
\Beq
L(j) \sspeq [L(L(j-1)_2)_1,\, L(L(j-1)_2)_2],\ \ \text{for}\ \ j\sspeq \range{1}{j_{\max}},
\Eeq
with\ input $L(0)\sspeq [X_N,\,k(N)]$, and $j_{\max}$ is defined by $L(j_{\max})\sspeq [B,\,0]$.
\pn 
Then the word is $w(N)\sspeq \cProd_{j=0}^{j_{\max}}\, L(j)_1$ (a concatenation product), and read as iterations acting on $0$ this becomes the representation $(N)_{ABC}$. The length of the word $w(N)$ is $j_{\max}\sspp 1$. \psn
\hskip 16.1cm $\square$
\pn
{\bf Example 4} $N\sspeq 38$. $L(0)\sspeq [A,\,11]$, $L(1)\sspeq [B,\,6]$, $L(2)\sspeq [B,\, 3]$, $L(3)\sspeq [C,\,0]$, and $L(4)\sspeq [B,0]$, hence $j_{\max}(38)\sspeq 4$, $w(38)\sspeq ABBCB$, and $(38)_{ABC}\sspeq ABBCB(0)$, to be read as $A(B(B(C(B(0)))))$.
\pbn
After these preliminaries the main theorem can be stated.
\begin{theorem*}
The tribonacci-representation $ZT(N)$ of  eq.~\ref{NT}, is equivalent to the tribonacci $ABC$-representation $(N)_{ABC}$ eq.~\ref{NABC},  for $N\sspin \mathbb N$. \Ie given any binary word $ZT(N)$ of eq.~\ref{NT}, beginning with $1$ and no three consecutive $1$s, which represents a unique $N$ one finds the unique representation of this $N$ given by $(N)_{ABC}$ without knowing the actual value of $N$; and {\it vice versa}, given a word $(N)_{ABC}$, representing uniquely a number $N$ one finds the representation of this $N$ as a binary word $ZT(N)$.   
\end{theorem*}
\psn
{\bf Proof:}\psn
Part {\bf I)}: The proof of the map $ZT(N)\sspto (N)_{ABC}$ is performed in three steps:
\psn
\Beqarray
\text{Step\ A)}:\ \ &&ZT(N)\ \ \sspto \widehat{ZT}(N)\sspdef 0(ZT(N)_{\text{reverse}})0\,, \nonumber\\ 
\text {Step\ B)}:\ \ &&\widehat{ZT}(N)\ \ \sspto  (N)_{AB\bullet\boldsymbol{\times}}\, , \nonumber\\
\text {Step\ C)}:\ \  &&(N)_{AB\bullet\boldsymbol{\times}}\sspto (N)_{ABC}\,. \nonumber 
\Eeqarray
These steps have been explained in {\it section 2}, part I) in detail. Step A) is clear. In step B) the six substitution rules of {\it Definition}~\ref{Subsrules} are important. They lead to the result that in $(N)_{AB\bullet\boldsymbol{\times}}$ only certain combinations like $A\bullet$, $\boldsymbol{\times\times}\bullet$, and at the end of the representation $AB$ or $\boldsymbol{\times\times}B$ can occur, besides any powers of $B$, but only one $B$ at the end. 
This $(N)_{AB\bullet\boldsymbol{\times}}$ words can also be read off from the $ABCTree_n$, given for  
$n\sspeq5$ in {\it Figure 3}, read from bottom to top with certain rules to choose the edge labels given in {\it section 2}. Step C) is again simple, because only the mentioned combinations can appear in $(N)_{AB\bullet\boldsymbol{\times}}$. 
\pn
Part {\bf II)}: The proof of the map $(N)_{ABC} \sspto ZT(N)$ is performed in three steps:
\psn
\Beqarray
{\text Step\ {\overline C})}:\ \  && (N)_{ABC}\ \ \sspto (N)_{AB\bullet\boldsymbol{\times}}\,, \nonumber\\ 
{\text Step\ {\overline B})}:\ \ &&(N)_{AB\bullet\boldsymbol{\times}}\ \ \sspto \widehat{ZT}(N)\,, \nonumber\\
{\text Step\ {\overline A})}:\ \  &&\widehat{ZT}(N)\ \ \sspto ZT(N)\,. \nonumber
\Eeqarray
Step $\overline C)$ uses the specific occurrences of $A\bullet$, $\boldsymbol{\times\times}\bullet$, and special endings mentioned above. Step $\overline B)$ is given in two equivalent versions. This is based on the six substitution rules ~\ref{Subsrules} in the reverse direction, using again the special combinations in which $A$ and $\boldsymbol{\times\times}$ appear. Step $\overline C)$ is again trivial.
\hskip 14.6cm $\square$\ 
\pbn
\pbn
\section{\bf Investigation of the $\bf A$,\,$\bf B$\, and $\bf C$ sequences}
\ps
In this section a detailed investigation of the $A,\,B\,$ and $C$ sequences is presented.
Some of these results can be found in references \cite{Carlitz}, \cite{Barcucci}, in \cite{OEIS} and elsewhere, but here the emphasis is on a derivation based on the infinite ternary tribonacci word $\TWord$ written as a sequence $t\sspeq$\seqnum{A080843} (see also {\it Table 1}). Its self-similarity leads to the following {\it Definition}, {\it Lemma} and six {\it Propositions}.
\pn
\begin{definition} \label{tword}
The tribonacci words $tw(l)$ over the alphabet $\{0,\,1,\, 2\}$ of length $\#tw(l) = T(l+2)$ are defined recursively by concatenations (we omit the concatenation symbol $\circ$) as
\Beq
tw(l) \sspeq tw(l-1)\,tw(l-2)\,tw(l-3),\ \ \  \text{with}\ \ tw(1)\sspeq 0,\, tw(2)\sspeq 01,\, tw(3)\sspeq 0102\,.  
\Eeq
\end{definition}
\noindent
Also $tw(0)\sspeq 2$ is used. \pn
The substitution map acting on tribonacci words and other strings with characters $\{0,\,1,\,2\}$ is defined  as a concatenation homomorphism by $\sigma\ :\ 0 \sspmapsto 01,\ 1\sspmapsto 02,\ 2\sspmapsto 0$. The inverse map is $\sigma^{[-1]}$ (One replaces first each $01$ and $02$ then the left over $0$). With $\sigma$ the words $tw(l)$ are generated iteratively from $tw(0) \sspeq 2$. $\sigma(tw(l))\sspeq tw(l+1)$, for $l\sspin \mathbb N_0$, and \dstyle{\lim_{l\to \infty} \sigma^{[l]}(0) \sspeq \TWord}. Self-similarity of $\TWord$ means $\sigma(\TWord) \sspeq \TWord$. \pn
Substrings of $\TWord$ of length $n$, starting with the first letter (number) $t(0) = 0$, are denoted by $s_n\sspdef \cProd_{j=0}^{n-1} t(n)$. If $n\sspeq T(l+2)$, for $l\sspin \mathbb N_0$, then $s_n = tw(l)$ (the string becomes a tribonacci word), and the numbers of $s_n$ map to the node labels of the last level of $TTree_{l}$ read from the \lhs. \pn
Also substrings of $\TWord$ not starting with $t(0)$ are used, like $\hat s_2\sspeq 02\sspeq \sigma(1)$, starting with $t(2)$.  \psn
In the following {\it Lemma} four definition of infinite words $t_i$, for $i\sspeq \range{1}{4}$, based on certain substrings $s_n$ (not all of them are tribonacci words) are given which are proved to be different partitions of $TWord$. This result will be used in the proofs of the following first {\it Propositions}~\ref{Prop10},~\ref{Bseq} and~\ref{Prop13}. 
\begin{lemma} \label{Partitions}\pbn 
\pbn
{\bf A)} With $s_{13}\sspeq 0102010010201\sspeq tw(5)$, $s_{11}\sspeq 01020100102$ and $s_7 \sspeq 0102010 \sspeq tw(4)$ define
\Beq\label{self1}
t_1\sspeq s_{13}s_{11}s_{13}s_{7}s_{13}s_{11}s_{13}s_{13}s_{11}s_{13}s_{7}s_{13}... \sspeq \cProd_{j=0}^{\infty}\, s_{\varepsilon(t(j))}, 
\Eeq 
where $\varepsilon(0)\sspeq 13$, $\varepsilon(1)\sspeq 11$ and $\varepsilon(2)\sspeq 7$.\psn
{\bf B)} With $s_7\sspeq 0102010\sspeq tw(4)$, $s_6\sspeq 010201$ and $s_4 \sspeq 0102\sspeq tw(3)$ define
\Beq\label{self2}
t_2\sspeq s_7s_6s_7s_4s_7s_6s_7s_7s_6s_7s_4s_7... \sspeq \cProd_{j=0}^{\infty}\, s_{\pi(t(j))}, 
\Eeq 
where $\pi(0)\sspeq 7$, $\pi(1)\sspeq 6$ and $\pi(2)\sspeq 4$.\psn
{\bf C)} With $s_4\sspeq 0102\sspeq tw(3)$, $s_3\sspeq 010$ and $s_2 \sspeq 01\sspeq tw(2)\sspeq \sigma(0)$ define
\Beq \label{self3}
t_3\sspeq s_4s_3s_4s_2s_4s_3s_4s_4s_3s_4s_2s_4... \sspeq \cProd_{j=0}^{\infty}\, s_{\tau(t(j))}, 
\Eeq
where $\tau(0)\sspeq 4$, $\tau(1)\sspeq 3$ and $\tau(2)\sspeq 2$.\psn
{\bf D)}
With $s_2\sspeq 01$, $\hat s_2\sspeq 02$ and $s_1 \sspeq 0\sspeq tw(1)\sspeq \sigma(2)$ define
\Beq \label{self4}
t_4\sspeq s_2\hat s_2s_2s_1s_2\hat s_2s_2s_2\hat s_2s_2s_2s_1...  
\Eeq
Here the string follows $t$ with $s_2$, $\hat s_2$ and $s_1$ playing the role of $0,\,1$ and $2$, respectively.\pn
Then
\Beq
t_1 \sspeq t_2 \sspeq t_3 \sspeq t_4\sspeq \TWord\,.
\Eeq
\end{lemma}
\noindent
{\bf Proof}:\pn
For $t_4$ from {\bf D}: The definition of $\sigma^{[-1]}$ shows that $\sigma^{[-1]}(t_4)\sspeq \TWord$. Hence $t_4\sspeq \sigma(\TWord)\sspeq \TWord$. \pn
For $t_3$ from {\bf C}: Because $\sigma(s_2)\sspeq s_4$, $\sigma(\hat s_2)\sspeq s_3 $ and 
$\sigma(s_1) \sspeq s_ 2$ it follows that $t_3\sspeq \sigma(t_4) \sspeq \TWord$.\pn
For $t_2$ from {\bf B}: Because $\sigma(s_4)\sspeq s_7$, $\sigma(s_3)\sspeq s_6 $ and $\sigma(s_2) \sspeq s_ 4$ it follows that $t_2\sspeq \sigma(t_3) \sspeq \TWord$.\pn
For $t_1$ from {\bf A}: Because $\sigma(s_7)\sspeq s_{13}$, $\sigma(s_6)\sspeq s_{11} $ and $\sigma(s_4) \sspeq s_ 7$ it follows that $t_1\sspeq \sigma(t_2)$ \pn
$\sspeq \TWord$.\pn
$\square$\psn
Using eq.\ref{self3} a formula for sequence entry $A(n)\sspeq$\seqnum{A278040}$(n)$ in terms of \dstyle{z(n)\sspdef \sum_{j=0}^n\,t(j)} is derived. This sequence $\{z(j)\}_{j=0}^{\infty}$ is given in \seqnum{A319198}. 
\begin{proposition} \label{Prop10}
\Beq \label{An}
A(n) \sspeq 4\,n\sspp 1 \sspm z(n-1), \ \text{for}\ n\sspin \mathbb N_0,  \ \text{with}\ \ z(-1) \sspeq 0\,.
\Eeq
\end{proposition}
\noindent
{\bf Proof:}\pn 
Define $\triangle A(k+1)\sspdef A(k+1)\sspm A(k)$. Consider the word $t_3$ of eq.~\ref{self3}.
The distances between the $1$s in the pairs $s_4s_3$, $s_3s_4$, $s_4s_2$, $s_2s_4$ and $s_4s_4$ are $4,\,3,\,4,\,2,\,4$. Therefore, the sequence of these distances is $4,\,3,\,4,\,2,\,4,\,3,\,4,\,4,\,3,\,4,\,2,\,...$. Thus, because the $s$-string $t_3$ follows the pattern of $t$, \ie of $\TWord$,    
\Beq \label{triangleA}
\triangle A(k+1)\sspeq 4\sspm t(k)\,, \ \ \text{for} \ \ k\sspeq \rangeinf{0}\,.
\Eeq
Then the telescopic sum produces the assertion, using $A(0)\sspeq 1$.\psn
\Beq
A(n)\sspeq A(0) \sspp \sum_{k=0}^{n-1}\, \triangle A(k+1)\sspeq 1 + 4\,n\sspm z(n-1),\ \ \text{with}\ \ z(-1)\sspeq 0.
\Eeq\pn
\hskip 16.1cm $\square$
\psn
The $B$ numbers \seqnum{A278039}, giving the increasing indices $k$ with $t(k) \sspeq 0$, come in three types: $B0$ numbers form the sequence of increasing indices $k$ of sequence $t$ with $t(k) \sspeq 0\sspeq t(k+1)$. Similarly the $B1$ sequence lists the increasing indices $k$ with $t(k) \sspeq 0,\, t(k+1)\sspeq 1$ and for the $B2$ sequence the indices $k$ are such that $t(k) \sspeq 0,\, t(k+1)\sspeq 2$.\pn
These numbers $B0(n)$, $B1(n)$ and $B2(n)$ are given by \seqnum{A319968}$(n+1)$, \seqnum{A278040}$(n) \sspm 1$, and \seqnum{A278041}$(n) \sspm 1$, respectively.\psn 
Before giving proofs we define the counting sequences $z_A(n)$, $z_B(n)$ and $z_C(n)$ to be the numbers of $A$, $B$ and $C$ numbers not exceeding $n\sspin \mathbb N$, respectively. If these counting functions appear for $n\sspeq -1$ they are set to $0$.\psn
These sequences are given by \seqnum{ A276797}$(n+1)$, \seqnum{A276796}$(n+1)$ and \seqnum{A276798}$(n+1)\sspm 1$ for $n\sspgeq -1$.\pn
Obviously,
\Beq \label{zn}
z(n)\sspeq 1\,z_A(n) \sspp 0\,z_B(n) \sspp 2\,z_C(n)\sspeq z_A(n)\sspp 2\,z_C(n),\ \ \text{for}\ \ n \sspeq -1,\,0,\,1,\,...\, . 
\Eeq 
These counting functions are obtained by partial sums of the corresponding characteristic sequences for the $A$, $B$ and $C$ numbers (or $0$, $1$, and $2$numbers in $t$), called $k_A$, $k_B$ and $k_C$, respectively.
\Beq \label{zXnsum}
z_X(n)\sspeq \sum_{k=0}^n\,k_X(k),\ \ \text{for} \ \ X\sspin \{A,\,B,\,C\}\ .  
\Eeq
The characteristic sequences members $k_A(n)$, $k_B(n)$ and $k_C(n)$ are given in \seqnum{A276794}$(n+1)$,  \seqnum{A276793}$(n+1)$ and \seqnum{A276791}$(n+1)$, for $n\sspin\mathbb N_0$, and they are, in terms of $t$, obviously given by
\Beqarray 
k_A(n)&\sspeq& t(n)\,(2\sspm t(n)), \label{kAn}\\ 
k_B(n)&\sspeq& \frac{1}{2}\,(t(n)\sspm 1)\,(t(n)\sspm 2), \label{kBn}\\
k_C(n)&\sspeq& \frac{1}{2}\,t(n)\,(t(n)\sspm 1). \label{kCn}
\Eeqarray
By definition it is trivial that (note the offset $0$ of the  $A,\,B,\,\,C$ sequences) 
\Beq \label{zXX}
z_X(X(k)) \sspeq k\sspp 1, \ \ \text{for}\ \ X\sspin \{A,\,B,\,C\}\ \ \text{and}\ \ k \sspin\mathbb N\,.
\Eeq
\begin{proposition} \label{Bseq}\psn
\pbn
For $n \sspin \mathbb N_0:$
\Beqarray
{\bf B0)}\  B0(n)&\sspeq& 13\,n \sspp 6 \sspm 2\,[z_A(n-1) \sspp 3\,z_C(n-1)] \sspeq 2\,C(n)\sspm n, \label{B0}\\
{\bf B1)}\ B1(n)&\sspeq& 4\,n\sspm z(n-1)\sspeq 4\,n\sspm [z_A(n-1) \sspp 2\,z_C(n-1)]\sspeq A(n)\sspm 1, \label{B1}\\
{\bf B2)}\  B2(n)&\sspeq& 7\,n\sspp 2\sspm [z_A(n-1) \sspp 3\,z_C(n-1)]\sspeq  \frac{1}{2}\,\left(B0(n)\sspp n\sspm 2\right) \nonumber\\
&\sspeq& C(n)\sspm 1, \label{B2}\\
{\bf B)}\, \ \ B(n)&\sspeq& 2\,n\sspm z_C(n-1)\, . \label{B}
\Eeqarray
For $n\sspeq 0$ empty sums in $z_A$, $z_C$ and $z$ are set to 0.
\end{proposition}
\psn
{\bf Proof:}\pn
{\bf B0}: Part 1:\ Define $\triangle B0(k+1)\sspdef B0(k+1)\sspm B0(k)$ and consider the word $t_1$ of eq.~\ref{self1}. The distances between pairs of $00$ in $s_{13}s_{11}$, $s_{11}s_{13}$, $s_{13}s_7$, $s_7s_{13}$ and $s_{13}s_{13}$ are $13,\,11,\,13,\,7,\,13$. Note that $s_7$ has no substring $00$, however because $s_7$ is always followed by $s_{13}$ the last $0$ of $s_7$ and the first of $s_{13}$ build the $00$ pair. Similarly, in the $s_{13}s_7$ case the last $0$ of $s_7$ is counted as a beginning of a $00$ pair. Therefore, the sequence of these distances is $13,\,11,\,13,\,7,\,13,\,11,\,13,\,13,\,11,\,13,\,7,\,...$. Because the $s$-string $t_1$ follows the pattern of $t$ the defect from $13$ is $0,\, -2,\, -6$ if $t(k) = 0,\,1,\,2$, hence     
\Beq
\triangle B0(k+1)\sspeq 13\sspm t(k)\,(t(k)\sspp 1)\,, \ \ \text{for} \ \ k\sspin \mathbb N_0\,.
\Eeq
The telescopic sum gives for $n\sspin \mathbb N$, with $B0(0)\sspeq 6$, 
\Beqarray
B0(n)&\sspeq& B0(0)\sspp \sum_{k=0}^{n-1}\,\triangle\, B0(k+1) \nonumber \\
&\sspeq& 6\sspp 13\,n \sspm 2\,(z_A(n-1)\sspp 3\,z_C(n-1))\,.
\Eeqarray
In the last step eqs.~\ref{zXnsum}, ~\ref{kAn} and ~\ref{kCn} have been used. This proves the first part of {\bf B0}. The proof of part 2 follows later from {\bf B2}.
\psn
{\bf B1}: With $\triangle B1(k+1)\sspdef B1(k+1)\sspm B1(k)$ and $t_3$ of eq.~\ref{self3} one finds for the distances between occurrences of $01$s similar to the above argument
\Beq \label{triangleB1}
\triangle B1(k+1)\sspeq 4\sspm t(k)\,, \ \ \text{for} \ \ k\sspin \mathbb N_0\,.
\Eeq
The telescopic sum gives, for $n\sspin \mathbb N$, with $B1(0) = 0$, 
\Beq 
B1(n)\sspeq 4\,n \sspm z(n-1), \nonumber
\Eeq
the first part of {\bf B1}, which shows, with eq~\ref{An}, also the third one.  The second part uses eq.~\ref{zn}.\psn
Note that $B1(n) = A(n)\sspm 1$ is trivial because $1$ in the tribonacci word $\TWord$ can only come from the substitution $\sigma(0)\sspeq 01$, and $\TWord$ (and $t$) starts with $0$. Therefore, one could directly prove {\bf B1} from eqs.~\ref{An} and ~\ref{zn} without first computing $\triangle B1(k+1)$. 
\psn
{\bf B2}: Because $2$ in $\TWord$ appears only from $\sigma(1)\sspeq 02$, it is clear that $B2(n)\sspeq C(n)\sspm 1$. Now one finds a formula for $C$ by looking first at $\triangle\,C(k+1)\sspdef C(k+1)\sspm C(k)$ using $t_2$ of eq.~\ref{self2}. The distances between consecutive $2$s in the five pairs $s_7s_6$, $s_6s_7$, $s_7s_4$, $s_4s_7$ and $s_7s_7$ is $7,\,6,\, 7,\,4,\,7$, respectively, and 
\Beq \label{triangleC}
\triangle \,C(k+1)\sspeq 7\sspm \frac{1}{2}\,t(k)\,(t(k)\sspp 1)\,, \ \ \text{for} \ \ k\sspin \mathbb N_0\,.
\Eeq
The telescopic sum leads here, using $C(0)\sspeq 3$, $z(n-1)$ from eq.~\ref{zn}  to 
\Beq \label{Cn}
C(n)\sspeq 7\,n\sspp 3 \sspm [z_A(n-1)\sspp 3\,z_C(n-1)]\,, \ \ \text{for} \ \ k\sspin \mathbb N_0\,.
\Eeq
This proves $\bf B2$, and also the second part of $\bf B0$. 
 \psn
{\bf B)}: Here $t_4$ of eq.~\ref{self4} can be used. The differences of $0$s in the five pairs $s_2\hat s_2$, $\hat s_2s_2$, $s_2s_1$, $s_1s_2$ and $s_2s_2$ is $2,\,2,\,2,\,1,\,2$. Thus
\Beq \label{triangleB}
\triangle\,B(k+1)\sspdef B(k+1)\sspm B(k)\sspeq 2\sspm \frac{1}{2}\,t(k)\,(t(k)\sspm 1)\sspeq 2\sspm k_C(n)\,, \ \ \text{for} \ \ k\sspin \mathbb N_0\,. 
\Eeq
In the last expression eq.~\ref{kCn} has been used. By telescoping, using $B(0)\sspeq 0$ and the definition of $z_C(n-1)$ from~\ref{zXnsum} proves the assertion. \hskip 8.5cm $\square$
\pbn
Eqs.~\ref{triangleC} and ~\ref{triangleB} show that $\triangle C(k+1) \sspm \triangle B(k+1) \sspeq 5\sspm t(k)$, for $k\sspin \mathbb N_0$. Telescoping leads to the result, obtained directly from eqs.~\ref{Cn} and ~\ref{B}, with eq.~\ref{zn},
\Beq
C(n)\sspm B(n) \sspeq 5\,n\sspp 3\sspm z(n-1)\,,\ \text{for} \ \ k\sspin \mathbb N_0\,, 
\Eeq
and with $A$ from eq.~\ref{An} this becomes
\Beq \label{ABCn}
C(n)\sspm (A(n)\sspp B(n))\sspeq n\sspp 2\,, \ \ \text{for} \ \ k\sspin \mathbb N_0\,. 
\Eeq 
This equation can be used to eliminate $C$ from the equations. 
\psn
Next the formulae for $z_X$ for $X\sspin \{A,\,B,\,C\}$ are listed, valid for $n\sspeq -1,\,0,\,1,\,... $.\psn
\begin{proposition} ~\label{Prop12}
\Beqarray
z_A(n)&\sspeq& 2\,B(n+1)\sspm A(n+1)\sspp 1\,, \label{zAn}  \\
z_B(n)&\sspeq& A(n+1)\sspm B(n+1)\sspm (n\sspp 2),\,\label{zBn}\\
z_C(n)&\sspeq& 2\,(n\sspp 1)\sspm B(n+1)\,. \label{zCn}
\Eeqarray
\end{proposition}
\psn
{\bf Proof:} Version 1. The inputs are $z_X(-1)\sspeq 0$, for $X\sspin \{A,\,B,\,C\}$, by definition, and are satisfied due to $B(0)\sspeq 0$ and $A(0)\sspeq 1$. Therefore $z_X(n) \sspeq \sum_{k=0}^n \triangle z_X(k)$, with $z_X(k)\sspdef z_X(k)\sspm z_X(k-1)$. The claimed formula and the known $\triangle A(k+1)$ and $\triangle B(k+1)$ from eqs.~\ref{triangleA} and ~\ref{triangleB}, respectively, produce the results $k_X(k)$ given in eqs.~\ref{kAn} to ~\ref{kCn}. Therefore  $z_X(n)$ coincides with eq. \ref{zXnsum}.\psn
Version 2. Besides eq.~\ref{zn} the trivial formula
\Beq \label{sumzX}
z_A(n)\sspp z_B(n) \sspp z_C(n) \sspeq n\sspp 1  
\Eeq
can be used.\pn
$z_A(n-1)$ is computed from the difference of  $3\,(z_A(n-1)\sspp 2\,z_C(n-1))$ from eq.~\ref{B1} and $2\,(z_A(n-1)\sspp 3\,z_C(n-1))$ from eq.~\ref{B2}, with $C(n)$ from eq.~\ref{ABCn}. This difference leads to the claim eq.~\ref{zAn} with $n\sspto n+1$.\pn
The $z_C(n)$ formula is eq.~\ref{B} with $n\sspto n+1$.\pn
$z_B(n)$ can then be computed from eq.~\ref{sumzX}. \hskip 8,5cm $\square$
\pb
Next all formulae for compositions of the types $X(Y(k)+1)$ and  $X(Y(k))$, for $X,\,Y\sspin \{A,\,B,\,C\}$ and $k\sspin \mathbb N_0$ shall be given. They are of interest in connection with the tribonacci $ABC$ representation given in the preceding section. For this, one needs first the results for the compositions $z(X(k))$. The formulae will be given in terms of $A$ and $B$ (with $C$ eliminated by eq. ~\ref{ABCn}).
\begin{proposition} ~\label{Prop13}
\Beqarray
z(A(k)) &\sspeq& 2\,(A(k)\sspm B(k))\sspm k \sspm 1, \label{zAk} \\
z(B(k)) &\sspeq& -A(k)\sspp 3\, B(k)\sspm k \sspp 1, \label{zBk} \\
z(C(k)) &\sspeq& B(k)\sspp 2\,k \sspp 3. \label{zCk}
\Eeqarray
\end{proposition}
\noindent
{\bf Proof:} $z(X(k))$ will be found from the self-similarity properties given in eqs.~\ref{self3}, ~\ref{self4} and ~\ref{self2}, for $X\sspeq A,\,B$ and $C$, respectively. These strings $t_3$, $t_4$ and $t_2$ are chosen because the relevant numbers $1,\, 0$ and $2$, respectively, appear precisely once in all $s-$substrings. For $z(X(k))\sspeq \sum_{j=0}^{X(k)}\, t(j)$ one has to sum all the numbers of the first $k$ substrings $s$ but in the last one only the numbers up to the number standing for $X$ are summed. \pn
A) In the $t_3$ substrings $s_4\sspeq 0102$, $s_3\sspeq 010$ and $s_2\sspeq 01$ the number $1$ appears just once. In all three substrings the sum up to the relevant number $1$ (for $A$) is $0\sspp 1\sspeq 1$, so for the last $s$ one has always to add $1$. Because $s_4$, $s_3$ and $s_2$, with sums $3,\, 1$ and $1$, play the role of $0,\,1$ and $2$, respectively, in $t_3$ one obtains $z(A(k))\sspeq 3\,z_B(k-1) +1\,(z_A(k-1)\sspp z_C(k-1)) \sspp 1$. With the identity eq.~\ref{sumzX} this becomes $2\,z_B(k-1)\sspp k\sspp 1$, and with the $z_B$ formula  eq.~\ref{zBn} this leads to the claim eq.~\ref{zAk}.\psn
B) In $t_4$ the sums of the substrings $s_2,\,\hat s_2,\, s_1$ are $1,\,2,\,0$,respectively,  and because all three begin with the relevant number $0$ nothing to be summed for the last $s$. Thus $z(B(k))\sspeq 1\,z_B(k-1) \sspp 2\,z_A(k-1)\sspp 0 \sspp 0$. Using eqs.~\ref{zBn} and ~\ref{zAn} this becomes the claim.\psn
C) In $t_2$ the sums are $4$ for $s_7,\, s_6$ and $3$ for $s_4$. The sums up to the relevant number $2$ are $3$ for each case. Therefore $z(C(k))\sspeq 4\,(z_B(k-1) \sspp z_A(k-1)) \sspp 3\,z_C(k-1)\sspp 3\sspeq z_B(k-1)\sspp z_A(k-1) \sspp 3\,k \sspp 3\sspeq B(k)\sspp 2\,k \sspp 3$, with eqs.~\ref{sumzX}, ~\ref{zBn} and  ~\ref{zAn}. \hskip 3cm $\square$
\pbn
\begin{proposition} \label{XYk}
\Beqarray 
A(A(k)+1) &\sspeq& 2\,(A(k)\sspp B(k))\sspp k\sspp 6\,, \hskip 1cm A(A(k))\sspeq  A(A(k)+1)\sspm 3\,, \label{AAk}\\
A(B(k)+1) &\sspeq& A(k)\sspp B(k) \sspp k\sspp 4\,,     \hskip 1.6cm A(B(k))\sspeq A(B(k)+1)\sspm 4\,, \label{ABk}\\
A(C(k)+1) &\sspeq& 4\,A(k)\sspp 3\,B(k)\sspp 2\,(k\sspp 5)\,,\hskip 0.5cm A(C(k))\sspeq A(C(k)+1)\sspm 2 \, .\label{ACk}
\Eeqarray
\Beqarray
B(A(k)+1) &\sspeq& A(k)\sspp B(k)\sspp k\sspp 3\,, \hskip 1.7cm B(A(k))\sspeq  B(A(k)+1)\sspm 2\,,\label{BAk}\\
B(B(k)+1) &\sspeq& A(k)\sspp 1\,,     \hskip 4.1cm B(B(k))\sspeq B(B(k)+1)\sspm 2\,, \label{BBk}\\
B(C(k)+1) &\sspeq& 2\,(A(k)\sspp B(k))\sspp k \sspp 5\,,\hskip 1.1cm B(C(k))\sspeq B(C(k)+1)\sspm 1 \, .\label{BCk}
\Eeqarray
\Beqarray
C(A(k)+1) &\sspeq& 4\,A(k)\sspp 3\,B(k) \sspp 2\,(k\sspp 6)\,, \hskip .5cm C(A(k))\sspeq  C(A(k)+1)\sspm 6\,,\label{CAk}\\
C(B(k)+1) &\sspeq& 2\,(A(k)\sspp B(k)) \sspp k\sspp 8\,,     \hskip 1.1cm C(B(k))\sspeq C(B(k)+1)\sspm 7\,, \label{CBk}\\
C(C(k)+1) &\sspeq& 7\,A(k)\sspp 6\,B(k)\sspp 4\,(k\sspp 5)\,,\hskip .6cm C(C(k))\sspeq C(C(k)+1)\sspm 4\, .\label{CCk}
\Eeqarray
\end{proposition}
\noindent
{\bf Proof:}\pn
The two versions are related by $\triangle X(n+1)\sspeq X(n+1)\sspm X(n)$ given in eqs.~\ref{triangleA},~\ref{triangleB} and~\ref{triangleC}, for $X\sspin \{A,\,B,\,C\}$, respectively, and $n$ replaced by $Y(k)$ with $Y\sspin \{A,\,B,\,C\}$.
For $C(n)$ eq.~\ref{ABCn} is always used.\pn
A) This follows from $A(n+1)$, given in eq.~\ref{An} with $n\sspto Y(k)$, $z(Y(k))$ from eqs.~\ref{zAk},~\ref{zBk}, and~\ref{zCk}.\psn
B) One proves that $B(A(k))\sspeq A(k)\sspp B(k) \sspp k \sspp 1$ from which $B(A(k)+1)$ follows. With eq.~\ref{ABCn} this means that one has to prove 
\Beq \label{BAkC}
B(A(k))\sspeqmust C(k)\sspm 1\sspeq B2(k). \nonumber
\Eeq
The second equality is \ref{B2}.
After applying $z_B$ on both sides, using eq.~\ref{zXX}, this is equivalent to
\Beq \label{zBCk}
A(k)\sspp 1\sspeqmust z_B(C(k)-1)) \sspeq z_B(C(k)). \nonumber
\Eeq
The second equality is trivial. This assertion is now proved. From eqs.~\ref{zBn},\,~\ref{An} and ~\ref{zCn} follows $z_B(n)\sspeq n\sspp 1 \sspm z(n)\sspp z_C(n)$.
Hence $z_B(C(k))\sspeq C(k)\sspp 1 - z(C(k))\sspp (k\sspp 1)$, with eq.~\ref{zXX}. This equals $C(k)\sspm k \sspm 1 \sspm B(k)$ from eq.~\ref{zCk}, and replacing $C(k)$ by eq. ~\ref{ABCn} gives $A(k) \sspp 1$.
\psn
One proves $B(B(k)) \sspeq A(k)\sspm 1$ or, after application of $z_B$ with eq.~\ref{zBn} on both sides, $B(k)\sspp 1 \sspeqmust z_B(A(k)\sspm 1)\sspeq z_B(A(k))$, where the second equality is trivial. But from eqs.~\ref{sumzX} and~\ref{zXX} follows $z_B(A(k))\sspeq A(k)\sspp 1 \sspm (k\sspp 1) \sspm z_C(A(k))$. Applying eq.~\ref{zCn} for $z_C(A(k))$, and the just proven eq.~\ref{BAk} for $B(A(k)+1)$ shows that
\Beq \label{zBAk}
z_B(A(k)) \sspeq B(k)\sspp 1.
\Eeq
The $B(C(k))$ claim can be written in terms of $C$ from eqs.~\ref{ABCn} and ~\ref{B0} as 
\Beq \label{BCk}
 B(C(k))\sspeqmust 2\,C(k)\sspm k \sspeq B0(k)\,, \nonumber
\Eeq
where eq.~\ref{B0} has been repeated.
\pn
Indeed, eqs.~\ref{B} and~\ref{zXX} imply for $B(C(k))\sspeq 2\,C(k)\sspm z_C(C(k)-1)\sspeq 2\,C(k)\sspm (z_C(C(k)) \sspm 1)\sspeq 2\,C(k)\sspm k$. The second equality is trivial.
\psn
C) These claims follow with $C(n+1)$ from eq.~\ref{ABCn} after replacement $n\sspto Y(k)$, and the already proved formulae for $A(Y(k)+1)$ and $B(Y(k)+1)$.
\pn
\hskip 16cm $\square$
\pbn
The collection of the results for $z_X(Y(k))$ is, for $k\sspin \mathbb N_0$: \psn
\begin{proposition} 
\Beqarray
z_A(A(k))&\sspeq& k\sspp 1, \nonumber\\
z_A(B(k))&\sspeq& A(k) \sspm B(k) \sspm (k\sspp 1)\sspeq  z_C(A(k)),\nonumber \\
z_A(C(k))&\sspeq& B(k)\sspp 1.\label{zAY}\\ 
&&\nonumber \\
z_B(A(k))&\sspeq& B(k)\sspp 1\sspeq z_A(C(k)), \nonumber\\
z_B(B(k))&\sspeq& k\sspp 1,\nonumber \\
z_B(C(k))&\sspeq& A(k)\sspp 1.\label{zBY}\\
&& \nonumber \\
z_C(A(k))&\sspeq& A(k) \sspm B(k)\sspm (k\sspp 1)\sspeq z_A(B(k)), \nonumber\\
z_C(B(k)&\sspeq& 2\,B(k)  \sspm A(k) \sspp 1,\nonumber \\
z_C(C(k))&\sspeq& k\sspp 1.\label{zBY}
\Eeqarray
\end{proposition}
\noindent
{\bf Proof}: \psn
That $z_X(X(k))\sspeq k\sspp 1$ has been noted already in eq.~\ref{zXX}.\psn
The other claims follow from the $z_X(n)$ results in {\it Proposition}~\ref{Prop12} after replacing $n$ by $Y(k)\sspneq X(k)$, and application of the formulae from {\it Proposition}~\ref{XYk}.  \hskip 8cm  $\square$
\pbn
As above mentioned many of the formulae of this section appear in \cite{Carlitz} and \cite{Barcucci} with the above mentioned translation between their sequences $a,\, b,\,$ and $c$ to our $B,\,A,$ and $C$. For example, {\it Theorem 13} of \cite{Carlitz}, p. 57, for the nine twofold iterations (in our notation $X(Y(k))$ of {\it Proposition}~\ref{XYk}) can be checked. 
\pbn
\pbn
{\bf Acknowledgment:} Thanks go to {\sl Neil Sloane} for an e-mail motivating this study. The author is grateful to an unknown referee of the first version who gave plenty of comments, suggestions and some questions concerning the proof of the theorem which led to this revision. 
\pbn 

\pbn 
\pbn
\hrule \psn
2010 Mathematics Subject Classification: Primary 11Y55; Secondary 32H50.
\psn 
{\em Keywords:}  Tribonacci numbers, tribonacci constant, tribonacci word,  tribonacci tree, tribonacci ABC-sequences, tribonacci ABC-tree. \psn
\hrule
\psn 
Concerned with OEIS \cite{OEIS} sequences: \seqnum{A0000073},\, \seqnum{A000201},\, \seqnum{A001622},\,  \seqnum{A001590},\, \seqnum{A001950},\, \seqnum{A003144},\, \seqnum{A003145},\, \seqnum{A003146},\, \seqnum{A005614},\, \seqnum{A058265},\,
 \seqnum{A080843},\, \seqnum{A092782}\, \seqnum{A158919},\, \seqnum{A189921},\, \seqnum{A276791},\, \seqnum{A276793},\, \seqnum{A276794},\, \seqnum{A276796},\,  \seqnum{A276797},\, \seqnum{A276798},\, \seqnum{A278038},\, \seqnum{A278039},\, \seqnum{A278040},\,  \seqnum{A278041},\, \seqnum{A278044},\, \seqnum{A316174},\, \seqnum{A316711},\, \seqnum{A316712},\,  \seqnum{A316713},\, \seqnum{A316714},\, \seqnum{A316715},\, \seqnum{A316716},\,  \seqnum{A316717},\, \seqnum{A317206},\, 
 \seqnum{A319195},\,  \seqnum{A319198},\, \seqnum{A319968}.
\psn
\hrule\psn
\vfill
\eject
\begin{landscape}
\begin{center}
{\large {\bf Table 1: Sequences $\bf t,\,A,\,B,\,C$, {\bf for}  $\bf n\sspeq \range{0}{79}$}} 
\end {center}
\begin{center}  
\begin{tabular}{|c|r|r|r|r|r|r|r|r|r|r|r|r|r|r|r|r|r|r|r|r|}\hline
&& && && && &&  && && && && &&\\
$\bf{n}$& $0$ & $1$ & $2$ & $3$ & $4$ & $5$ & $6$ & $7$ & $8$ & $9$ & $10$ & $11$ & $12$ & $13$ & $14$ & $15$ & $16$ & $17$ & $18$ & $19$\\ 
\hline
$\bf t$& $0$ & $1$ & $0$ & $2$ & $0$ & $1$ & $0$ & $0$ & $1$ & $0$ & $2$ & $0$ & $1$ & $0$ & $1$ & $0$ & $2$ & $0$ & $1$ & $0$\\
\hline
$\bf A$& $1$ & $5$ & $8$ & $12$ & $14$ & $18$ & $21$ & $25$ & $29$ & $32$ & $36$ & $38$ & $42$ & $45$ & $49$ & $52$ & $56$ & $58$ & $62$ & $65$\\
\hline
$\bf B$& $0$ & $2$ & $4$ & $6$ & $7$ & $9$ & $11$ & $13$ & $15$ & $17$ & $19$ & $20$ & $22$ & $24$ & $26$ & $28$ & $30$ & $31$ & $33$ & $35$\\
\hline
$\bf C$& $3$ & $10$ & $16$ & $23$ & $27$ & $34$ & $40$ & $47$ & $54$ & $60$ & $67$ & $71$ & $78$ & $84$ & $91$ & $97$ & $104$ & $108$ & $115$ & $121$\\
\hline
\hline
$\bf{n}$& $20$ & $21$ & $22$ & $23$ & $24$ & $25$ & $26$ & $27$ & $28$ & $29$ & $30$ & $31$ & $32$ & $33$ & $34$ & $35$ & $36$ & $37$ & $38$ & $39$\\ 
\hline
$\bf t$& $0$ & $1$ & $0$ & $2$ & $0$ & $1$ & $0$ & $2$ & $0$ & $1$ & $0$ & $0$ & $1$ & $0$ & $2$ & $0$ & $1$ & $0$ & $1$ & $0$\\
\hline
$\bf A$& $69$ & $73$ & $76$ & $80$ & $82$ & $86$ & $89$ & $93$ & $95$ & $99$ & $102$ & $106$ & $110$ & $113$ & $117$ & $119$ & $123$ & $126$ & $130$ & $133$\\
\hline
$\bf B$& $37$ & $39$ & $41$ & $43$ & $44$ & $46$ & $48$ & $50$ & $51$ & $53$ & $55$ & $57$ & $59$ & $61$ & $63$ & $64$ & $66$ & $68$ & $70$ & $72$\\
\hline
$\bf C$& $128$ & $135$ & $141$ & $148$ & $152$ & $159$ & $165$ & $172$ & $176$ & $183$ & $189$ & $196$ & $203$ & $209$ & $216$ & $220$ & $227$ & $233$ & $240$ & $246$\\
\hline
\hline
$\bf{n}$& $40$ & $41$ & $42$ & $43$ & $44$ & $45$ & $46$ & $47$ & $48$ & $49$ & $50$ & $51$ & $52$ & $53$ & $54$ & $55$ & $56$ & $57$ & $58$ & $59$\\ 
\hline
$\bf t$& $2$ & $0$ & $1$ & $0$ & $0$ & $1$ & $0$ & $2$ & $0$ & $1$ & $0$ & $0$ & $1$ & $0$ & $2$ & $0$ & $1$ & $0$ & $1$ & $0$\\
\hline
$\bf A$& $137$ & $139$ & $143$ & $146$ & $150$ & $154$ & $157$ & $161$ & $163$ & $167$ & $170$ & $174$ & $178$ & $181$ & $185$ & $187$ & $191$ & $194$ & $198$ & $201$\\
\hline
$\bf B$& $74$ & $75$ & $77$ & $79$ & $81$ & $83$ & $85$ & $87$ & $88$ & $90$ & $92$ & $94$ & $96$ & $98$ & $100$ & $101$ & $103$ & $105$ & $107$ & $109$\\
\hline
$\bf C$& $253$ & $257$ & $264$ & $270$ & $277$ & $284$ & $290$ & $297$ & $301$ & $308$ & $314$ & $321$ & $328$ & $334$ & $341$ & $345$ & $352$ & $358$ & $365$ & $371$\\
\hline
\hline
$\bf{n}$& $60$ & $61$ & $62$ & $63$ & $64$ & $65$ & $66$ & $67$ & $68$ & $69$ & $70$ & $71$ & $72$ & $73$ & $74$ & $75$ & $76$ & $77$ & $78$ & $79$\\ 
\hline
$\bf t$& $2$ & $0$ & $1$ & $0$ & $0$ & $1$ & $0$ & $2$ & $0$ & $1$ & $0$ & $2$ & $0$ & $1$ & $0$ & $0$ & $1$ & $0$ & $2$ & $0$\\
\hline
$\bf A$& $205$ & $207$ & $211$ & $214$ & $218$ & $222$ & $225$ & $229$ & $231$ & $235$ & $238$ & $242$ & $244$ & $248$ & $251$ & $255$ & $259$ & $262$ & $266$ & $268$\\
\hline
$\bf B$& $111$ & $112$ & $114$ & $116$ & $118$ & $120$ & $122$ & $124$ & $125$ & $127$ & $129$ & $131$ & $132$ & $134$ & $136$ & $138$ & $140$ & $142$ & $144$ & $145$\\
\hline
$\bf C$& $378$ & $382$ & $389$ & $395$ & $402$ & $409$ & $415$ & $422$ & $426$ & $433$ & $439$ & $446$ & $450$ & $457$ & $463$ & $470$ & $477$ & $483$ & $490$ & $494$\\
\hline
\end{tabular}
\end{center}
\end{landscape}
\vfill
\eject
\begin{center}
{\large {\bf Table 2: $\bf ZT(N)$, for $\bf N \sspeq \range{1}{100}$}} 
\end {center}
\begin{center}  
\begin{tabular}{|c|l||c|l||c|l||c|l||c|l|}\hline
&& && && && &\\
$\bf{N}$& $\bf ZT(N)$ &$\bf{N}$& $\bf ZT(N)$& $\bf{N}$& $\bf ZT(N)$ &$\bf{N}$& $\bf ZT(N)$&$\bf{N}$& $\bf ZT(N)$\\ 
&& && && && &\\ \hline\hline
\bf 1&  1    &\bf 21& 11001 &\bf 41& 110100 &\bf 61& 1010100  &\bf 81 & 10000000 \\
\hline
\bf 2&  10    &\bf 22& 11010 &\bf 42& 110101  &\bf 62& 1010101  &\bf 82 & 10000001 \\
\hline
\bf 3&  11     &\bf 23& 11011 &\bf 43& 110110 &\bf 63& 1010110   &\bf 83 & 10000010 \\
\hline
\bf 4& 100   &\bf 24& 100000 &\bf 44& 1000000 &\bf 64& 1011000 &\bf 84 &  10000011 \\
\hline
\bf 5& 101   &\bf 25& 100001 &\bf 45& 1000001 &\bf 65& 1011001 &\bf 85 & 10000100 \\
\hline
\bf 6& 110    &\bf 26& 100010  &\bf 46& 1000010&\bf 66& 1011010 &\bf 86 & 10000101 \\
\hline
\bf 7& 1000 &\bf 27& 100011 &\bf 47& 1000011 &\bf 67& 1011011 &\bf 87 & 10000110 \\
\hline
\bf 8& 1001  &\bf 28& 100100 &\bf 48& 1000100 &\bf 68& 1100000 &\bf 88 & 10001000 \\
\hline
\bf 9& 1010  &\bf 29& 100101  &\bf 49& 1000101 &\bf 69& 1100001 &\bf 89 & 10001001 \\
\hline
\bf 10& 1011   &\bf 30& 100110 &\bf 50& 1000110 &\bf 70& 1100010 &\bf 90 & 10001010 \\
\hline
\bf 11& 1100  &\bf 31& 101000 &\bf 51& 1001000 &\bf 71& 1100011 &\bf 91 & 10001011 \\
\hline
\bf 12& 1101    &\bf 32& 101001 &\bf 52& 1001001 &\bf 72& 1100100 &\bf 92 & 10001100 \\
\hline
\bf 13& 10000 &\bf 33& 101010 &\bf 53& 1001010&\bf 73& 1100101 &\bf 93 & 10001101 \\
\hline
\bf 14& 10001 &\bf 34& 101011 &\bf 54& 1001011 &\bf 74& 1100110 &\bf 94 & 10010000 \\
\hline
\bf 15& 10010  &\bf 35& 101100 &\bf 55& 1001100 &\bf 75& 1101000 &\bf 95 & 10010001 \\
\hline
\bf 16& 10011  &\bf 36& 101101 &\bf 56& 1001101 &\bf 76& 1101001 &\bf 96 & 10010010 \\
\hline
\bf 17& 10100  &\bf 37& 110000 &\bf 57& 1010000  &\bf 77& 1101010 &\bf 97 & 10010011 \\
\hline
\bf 18& 10101  &\bf 38& 110001&\bf 58& 1010001 &\bf 78& 1101011  &\bf 98 & 10010100 \\
\hline
\bf 19& 10110  &\bf 39& 110010 &\bf 59& 1010010 &\bf 79& 1101100 &\bf 99 & 10010101 \\
\hline
\bf 20& 11000  &\bf 40& 110011 &\bf 60& 1010011 &\bf 80& 1101101 &\bf 100 & 10010110 \\
\hline  
\hline
\end{tabular}
\end{center}
\vfill
\eject
\begin{center}
{\large {\bf Table 3: $\bf ABC(N)$, for $\bf N \sspeq \range{1}{100}$}} 
\end {center}
\begin{center}  
\begin{tabular}{|c|l||c|l||c|l||c|l||c|l|}\hline
&& && && && &\\
$\bf{N}$& $\bf ABC(N)$ &$\bf{N}$& $\bf ABC(N)$& $\bf{N}$& $\bf ABC(N)$ &$\bf{N}$& $\bf ABC(N)$&$\bf{N}$& $\bf ABC(N)$\\ 
&& && && && &\\ \hline\hline
\bf 1&  10    &\bf 21& 1020 &\bf 41& 00120 &\bf 61& 001110  &\bf 81 & 000000010 \\
\hline
\bf 2& 010    &\bf 22& 0120 &\bf 42& 1120 &\bf 62& 11110  &\bf 82 & 10000010 \\
\hline
\bf 3& 20     &\bf 23& 220 &\bf 43& 0220 &\bf 63& 02110   &\bf 83 & 01000010 \\
\hline
\bf 4& 0010   &\bf 24& 0000010 &\bf 44& 00000010 &\bf 64& 000210 &\bf 84 & 2000010 \\
\hline
\bf 5&  110   &\bf 25& 100010 &\bf 45& 1000010 &\bf 65& 10210 &\bf 85 & 00100010 \\
\hline
\bf 6& 020    &\bf 26& 010010 &\bf 46& 0100010 &\bf 66& 01210 &\bf 86 & 1100010 \\
\hline
\bf 7&  00010 &\bf 27& 20010  &\bf 47& 200010 &\bf 67& 2210 &\bf 87 & 0200010\\
\hline
\bf 8&  1010  &\bf 28& 001010 &\bf 48& 0010010 &\bf 68& 0000020 &\bf 88 & 00010010\\
\hline
\bf 9&  0110  &\bf 29& 11010 &\bf 49& 110010 &\bf 69& 100020 &\bf 89 & 1010010 \\
\hline
\bf 10& 210   &\bf 30& 02010 &\bf 50& 020010 &\bf 70& 010020 &\bf 90 & 0110010 \\
\hline
\bf 11& 0020  &\bf 31& 000110 &\bf 51& 0001010 &\bf 71& 20020 &\bf 91 & 210010 \\
\hline
\bf 12& 120    &\bf 32& 10110 &\bf 52& 101010 &\bf 72& 001020 &\bf 92 & 0020010 \\
\hline
\bf 13& 000010 &\bf 33& 01110 &\bf 53& 011010 &\bf 73& 11020 &\bf 93 & 120010 \\
\hline
\bf 14&  10010 &\bf 34& 2110 &\bf 54& 21010 &\bf 74& 02020 &\bf 94 & 00001010 \\
\hline
\bf 15& 01010  &\bf 35& 00210 &\bf 55& 002010 &\bf 75& 000120 &\bf 95 & 1001010 \\
\hline
\bf 16&  2010  &\bf 36& 1210 &\bf 56& 12010 &\bf 76& 10120 &\bf 96 & 0101010 \\
\hline
\bf 17& 00110  &\bf 37& 000020 &\bf 57& 0000110 &\bf 77& 01120 &\bf 97 & 201010 \\
\hline
\bf 18& 1110   &\bf 38& 10020 &\bf 58& 100110 &\bf 78& 2120  &\bf 98 & 0011010 \\
\hline
\bf 19& 0210   &\bf 39& 01020 &\bf 59& 010110 &\bf 79& 00220 &\bf 99 & 111010\\
\hline
\bf 20& 00020  &\bf 40& 2020 &\bf 60& 20110 &\bf 80& 1220 &\bf 100 & 021010 \\
\hline  
\hline
\hline
\end{tabular}
\end{center}
\noindent
Here $0, 1$ and $2$ stand for $B, A$ and $C$, respectively. \Eg $ABC(6)\sspeq 020$ to be read as $BCB \sspeq B(C(B(0)))\sspeq (6)_{ABC}$.

\end{document}